\documentclass[11pt,letterpaper,reqno]{amsart}
\usepackage{microtype}

\usepackage[shortlabels]{enumitem}
\setlist[enumerate]{label={\arabic*.}}

\usepackage{amssymb}
\usepackage{mathtools}
\usepackage[dvipsnames]{xcolor}
\usepackage
[colorlinks=true,linkcolor=Maroon,citecolor=OliveGreen,backref]
{hyperref}
\usepackage{bookmark}
\usepackage[abbrev,shortalphabetic]{amsrefs}  
\usepackage[capitalize]{cleveref} 
\usepackage{tikz}

\usepackage[norefs, nocites]{refcheck}

\makeatletter
\newcommand{\refcheckize}[1]{%
  \expandafter\let\csname @@\string#1\endcsname#1%
  \expandafter\DeclareRobustCommand\csname relax\string#1\endcsname[1]{%
    \csname @@\string#1\endcsname{##1}\wrtusdrf{##1}}%
  \expandafter\let\expandafter#1\csname relax\string#1\endcsname
}
\makeatother

\refcheckize{\cref}
\refcheckize{\Cref}

\numberwithin{equation}{section}

\newtheorem{theorem}{Theorem}[section]
\newtheorem*{theorem*}{Theorem}
\newtheorem{lemma}[theorem]{Lemma}
\newtheorem{proposition}[theorem]{Proposition}
\newtheorem{corollary}[theorem]{Corollary}

\theoremstyle{definition}
\newtheorem{remark}[theorem]{Remark}
\newtheorem{definition}[theorem]{Definition}

\newcommand{\eps}{\varepsilon}
\renewcommand{\phi}{\varphi}
\renewcommand{\bar}{\overline}

\renewcommand{\hat}{\widehat}
\renewcommand{\subset}{\subseteq}
\renewcommand{\supset}{\supseteq}

\def\Ell{\mathcal{L}}

\newcommand\opr[1]{\operatorname{#1}}

\def\R{\mathbf{R}}

\def\C{\mathbf{C}}
\def\Z{\mathbf{Z}}
\def\F{\mathbf{F}}
\def\E{\mathbf{E}}
\def\K{\mathbf{K}}

\def\GL{\opr{GL}}
\def\SL{\opr{SL}}
\def\PGL{\opr{PGL}}

\def\Aut{\opr{Aut}}

\def\E{\mathbf{E}}

\def\nsgp{\trianglelefteq}
\def\csgp{\mathrel{
\mathrlap{\trianglelefteq}{\raisebox{1pt}{$\blacktriangleleft$}}
}}
\def\tr{\opr{tr}}

\def\chr{\opr{char}}
\def\Sol{\opr{Sol}}
\def\Lie{\opr{Lie}}
\def\Sym{\opr{Sym}}

\def\Hom{\opr{Hom}}

\newcommand\br[1]{{\left(#1\right)}}
\newcommand\floor[1]{\left\lfloor{#1}\right\rfloor}

\newcommand{\gen}[1]{\langle{#1}\rangle}

\newcommand{\frat}{\opr{Frat}}
\newcommand{\roots}{\Phi}

\usepackage{todonotes}

\begin{document}

    \title{Growth in linear groups}

    \author[Eberhard]{Sean Eberhard}
    \address{Sean Eberhard, Mathematical Sciences Research Centre, Queen's University Belfast, Belfast BT7~1NN, UK}
    \email{s.eberhard@qub.ac.uk}

    \author[Murphy]{Brendan Murphy}
    \address{Brendan Murphy, Atmospheric Chemistry Research Group, School of Chemistry, University of Bristol, Bristol BS8~1TS, UK}
    \email{brendan.murphy@bristol.ac.uk}

    \def\hungaryaddress{{A.~R\'enyi Institute of Mathematics,
        E\"otv\"os Lor\'and Research Network,
        P.O.~Box~127, H-1364 Budapest, Hungary}}
    \author[Pyber]{L\'aszl\'o Pyber}
    \address{L\'aszl\'o Pyber, \hungaryaddress}
    \email{pyber@renyi.hu}

    \author[Szab\'o]{Endre Szab\'o}
    \address{Endre Szab\'o, \hungaryaddress}
    \email{endre@renyi.hu}

    \thanks{SE has received funding from the European Research Council (ERC) under the European Union’s Horizon 2020 research and innovation programme (grant agreement No. 803711) and from the Royal Society.
    BM was funded by The Leverhulme Trust through Leverhulme grant RPG 2017-371.
     LP was supported by the National Research,
    Development and Innovation Office (NKFIH) Grant K115799,
    ESz was supported by the NKFIH Grants K115799 and K120697.
    The project leading to this application has received funding from
    the European Research Council (ERC) under the European Union's
    Horizon 2020 research and innovation programme (grant agreement No
    741420)}

    \begin{abstract}
        We prove a conjecture of Helfgott on the structure of sets of bounded tripling in bounded rank,
        which states the following.
        Let $A$ be a finite symmetric subset of $\GL_n(\F)$ for any field $\F$ such that $|A^3| \leq K|A|$.
        Then there are subgroups $H \nsgp \Gamma \nsgp \langle A \rangle$ such that
        $A$ is covered by $K^{O_n(1)}$ cosets of $\Gamma$,
        $\Gamma / H$ is nilpotent of step at most $n-1$, and
        $H$ is contained in $A^{O_n(1)}$.
        This theorem includes the Product Theorem for finite simple
        groups of bounded rank as a special case.
        As an application of our methods we also show that
        the diameter of sufficiently quasirandom finite linear groups is poly-logarithmic.
    \end{abstract}

    \maketitle

    \setcounter{tocdepth}{1}
    \tableofcontents

    \section{Introduction}

    \subsection{Statement of results}

    In this paper we characterize sets of bounded tripling in $\GL_n(\F)$,
    where $n$ is bounded and $\F$ is an arbitrary field.
    Here a finite set $A \subset \GL_n(\F)$ is said to be \emph{$K$-tripling} if $|A^3| \leq K|A|$.
    This notion is largely the same as that of a finite \emph{$K$-approximate group},
    which is a symmetric set $A$ containing $1$ such that $A^2$ is covered by at most $K$ translates of $A$.
    Prototypical examples include subgroups and progressions $\{g^n : |n| \leq N\}$,
    as well as certain nilpotent generalizations of progressions called nilprogressions, such as the Heisenberg nilprogression
    \[
        A = \left\{ \begin{pmatrix} 1 & x & z \\ 0 & 1 & y \\ 0 & 0 & 1 \end{pmatrix} : |x|, |y| \leq N, |z| \leq N^2 \right\} \subset \GL_3(\R).
    \]
    In broad qualitative terms, the most general approximate group is an extension of a subgroup by a nilprogression:
    this is the content of the celebrated structure theorem for approximate groups proved by Breuillard, Green, and Tao~\cite{BGT}.

    \begin{theorem}[Breuillard, Green, Tao~\cite{BGT}]
        \label{BGT}
        Let $A$ be a $K$-approximate subgroup of any group $G$.
        Then there are subgroups $H \nsgp \Gamma \leq G$ with the following properties:
        \begin{enumerate}[(1)]
            \item $A$ is covered by $O_K(1)$ cosets of $\Gamma$,
            \item $\Gamma / H$ is nilpotent of rank and step at most $O_K(1)$,
            \item $H$ is contained in $A^4$.
        \end{enumerate}
    \end{theorem}

    Breuillard, Green, and Tao described the statement of \Cref{BGT} as the \emph{Helfgott--Lindenstrauss conjecture}.
    However, there is some inconsistency in the usage of this term.
    \Cref{BGT} corresponds closely to the essentially qualitative conjecture made by Lindenstrauss (Lindenstrauss, personal communication).
    On the other hand the conjecture of Helfgott~\cites{helfgott-SL3, tao-blog}
    predicted \emph{polynomial bounds}, at least for groups in \emph{bounded rank}.
    This was later formulated more precisely in \cite{gill-helfgott} as well as \cite{helfgott--perspectives}.

    Unfortunately the proof of \Cref{BGT} depends on nonstandard analysis as well as the solution to Hilbert's fifth problem,
    and there seems to be no way to deduce an explicit bound for the number of cosets of $\Gamma$ required to cover $A$
    (the bounds on the rank and step of $\Gamma / H$ are explicit and benign: see \cite{BGT}*{Remark~1.9}).
    However, it is now known~\cites{breuillard-tointon, eberhard-arxiv} (see also \Cref{sec:examples}) that the number of cosets required is not polynomial in general in high rank.

    In this paper we prove Helfgott's original conjecture.\footnote{
      This version of our paper differs significantly from the first
      version that appeared on the arXiv. In this version we prove the
      normality of $\Gamma$ in $G$, which requires significant
      generalization of our earlier methods. We had to extend
      results about subgroups to results about sections.
      See \Cref{sec:outline-paper} for more details.}
    That is, we give a version of the Breuillard--Green--Tao theorem with polynomial bounds in bounded rank, over arbitrary fields.

    \begin{theorem}
        \label{thm:main}
        Let $\F$ be a field of characteristic $p \geq 0$ and let $A \subset \GL_n(\F)$ be finite, nonempty, symmetric, and $K$-tripling, where $K \ge 2$.
        Let $G = \gen{A}$.
        Then there are subgroups $H \nsgp \Gamma \nsgp G$ such that
        \begin{enumerate}[(1)]
            \item $A$ is covered by $K^{O_n(1)}$ cosets of $\Gamma$,
            \item $\Gamma / H$ is nilpotent of step at most $n-1$,
            \item $H$ is contained in $A^{O_n(1)}$.
        \end{enumerate}
        Moreover, if $p > 0$ then $H$ has a perfect soluble-by-$\Lie^*(p)$ normal subgroup $P$ contained in a translate of $A^3$ such that $H/P$ is a $p$-group, and if $p = 0$ then $H$ is trivial.
    \end{theorem}

    Here we write $\Lie(p)$ for the class of finite simple groups of Lie type of characteristic $p$
    and we say $G$ is in $\Lie^*(p)$ if it is a finite direct product of members of $\Lie(p)$.
    For convenience if $p = 0$ then $\Lie(p)$ is defined to be empty and $\Lie^*(p)$ consists of just the trivial group.

    \begin{remark}\leavevmode
    \begin{enumerate}[(a)]
        \item
        Without loss of generality $H = \gamma_n(\Gamma)$, the $n$th term of the lower central series of $\Gamma$.
        Similarly $P$ must be the perfect core of $\Gamma$ (the last term of the derived series).
        Thus we may assume $H$ and $P$ are characteristic in $\Gamma$ and normal in $G$.
        \item
        The rank of $\Gamma / H$ cannot be controlled without sacrificing normality of $\Gamma$ in $G$.
        In fact in \Cref{sec:examples} we give an example in which $\Gamma$ cannot be chosen to be finitely generated.
        However if $\Gamma$ is not required to be normal then we may
        assume $\Gamma = \gen{A^6 \cap \Gamma}$ (see \Cref{lem:covering-observation}), and by applying the
        result of Tointon~\cite{tointon-nilpotent} we may choose $H
        \subset A^{K^{O_n(1)}}$ so that $P \le \gamma_n(\Gamma) \leq H \nsgp
        \Gamma$, $\Gamma / H$ has rank $K^{O_n(1)}$, and $A$ is
        still covered by $K^{O_n(1)}$ cosets of $\Gamma$.
    \end{enumerate}
    \end{remark}

    Roughly half the proof of \Cref{thm:main} consists of establishing the existence of $P$. This preliminary ``soluble version'' was previously announced in \cite{PS-survey}.

    \begin{theorem}
        \label{thm:main-soluble}
        Let hypotheses be as \Cref{thm:main}.
        Then there are subgroups $P \nsgp \Gamma \nsgp \langle A \rangle$ such that
        \begin{enumerate}[(1)]
            \item $A$ is covered by $K^{O_n(1)}$ cosets of $\Gamma$,
            \item $\Gamma / P$ is soluble of derived length $O(\log n)$,
            \item $P$ is perfect, soluble-by-$\Lie^*(p)$, and contained in a translate of $A^3$.
        \end{enumerate}
    \end{theorem}

    These results generalize and depend on the Product Theorem
    for finite simple groups, obtained independently by Breuillard,
    Green, and Tao~\cite{BGT-linear} and Pyber and Szab\'o~\cite{PS-JAMS}.

    \begin{theorem}[\cite{PS-JAMS}*{Theorem~2}, see also \cite{BGT-linear}*{Corollary~2.4}]
        \label{thm:PS-main}
        Let $L$ be a finite simple group of Lie type of rank $r$ and $A$ a generating set of $L$.
        Then either
        \begin{enumerate}[(1)]
            \item $|A^3| > |A|^{1 + \eps}$, where $\eps = \eps(r)$ depends only on $r$, or
            \item $A^3 = L$.
        \end{enumerate}
    \end{theorem}

    A key new ingredient in this paper is something we call the ``affine conjugating trick''.
    See the outline below for a description.
    As a further application of this trick, we prove a diameter bound for quasirandom groups.

    \begin{theorem} \label{thm:polylog-dimeter-intro}
        For each positive integer $n$ there are positive numbers
        $K=K(n)$ and $c=c(n)$
        with the following property.
        Let $\F$ be a field and $G\le\GL_n(\F)$ be a $K$-quasirandom finite
        subgroup. Then the Cayley graph of $G$ with respect to any
        generating set has diameter at most $\br{\log|G|}^c$.
    \end{theorem}

    For $p$-generated perfect subgroups of
    $\SL_n(\F_p)$ this was proved in \cite{PS-JAMS}.
    Like the proof of \Cref{thm:main} in the special case of subsets of
    $\GL_n(\F_p)$
    (\cite{gill-helfgott} building on \cites{PS-arxiv,PS-JAMS}),
    the proof of this special case
    depends in an essential way  on the Nori
    correspondence between $p$-generated subgroups of $\SL_n(\F_p)$
    and certain closed subgroups of $\SL_n(\F_p)$.
    The fact that the affine conjugating trick can be used to replace
    the Nori correspondence in the proof of these two related results,
    and used in the proof of their extension to arbitrary fields, clearly shows
    the power of this trick.

    \subsection{Relation to previous literature}

    In characteristic zero, \Cref{thm:main} was first established by Breuillard, Green, and Tao~\cite{BGT-linear}.
    We include this case mainly for the sake of having a uniform argument for all fields, and because it is hardly any extra work, but it is not the important contribution of this paper.
    The fact that the arbitrary field case reduces to the case of finite fields is notable, and validates the opinion of Gill and Helfgott (see~\cite{gill-helfgott}*{Section~1.3}) that finite fields contain the heart of the matter.

    The prime finite field case of \Cref{thm:main} was proved by Gill and Helfgott~\cite{gill-helfgott},
    conditional on the prime finite field case of \Cref{thm:main-soluble},
    a result of the third and fourth authors which has not previously appeared in published form.
    It appears in an unpublished part of an earlier version \cite{PS-arxiv} of \cite{PS-JAMS} as Corollary 105,
    but the published version \cite{PS-JAMS} contains only a preliminary result, Lemma 73, on perfect $p$-generated subgroups of $\SL_n(\F_p)$ (Theorem 85 in \cite{PS-arxiv}).
    The publication of Corollary 105 has been deferred until now, in anticipation of the full version above applying to all fields uniformly.

    The results proved in \cites{PS-arxiv,PS-JAMS} depend in an essential way on the Nori correspondence between $p$-generated subgroups of $\SL_n(\F_p)$ and certain closed subgroups of $\SL_n(\overline{\F_p})$.
    Since the Nori theory does not extend to subgroups of $\SL_n(\F_q)$, $q$ a prime power, we had to devise an entirely different argument to prove \Cref{thm:main-soluble}.
    We use arguments based on quasirandomness and a new ``affine conjugating trick'', to be described below.

    Similarly, the method of \cite{gill-helfgott} depends on the theory of algebraic groups over $\F_p$, and several critical features of this theory break down for fields of prime-power order, such as boundedness for chains of unipotent subgroups.
    Accordingly, while parts of our method are inspired by \cite{gill-helfgott} (in particular their pivoting theorem is a key tool) our deduction of \Cref{thm:main} from \Cref{thm:main-soluble} uses many entirely new tools, such as a growth lemma for images of bilinear maps and a more general descent argument.
    In fact, the subgroup $\Gamma$ arises in a slightly different way in our method: we do not take a full root kernel but we use a tricky pigeonholing argument to identify an appropriate normal subgroup.

    \subsection{Outline of the paper} \label{sec:outline-paper}

    See \Cref{sec:notation} for notation. In \Cref{sec:examples} we give a few examples that illustrate some of the subtleties of \Cref{thm:main} and \Cref{thm:main-soluble}.

    \Cref{sec:toolbox} recalls some tools that may be familiar to experts:
    basic group theory and arithmetic combinatorics (or nonabelian additive combinatorics), quasirandomness, and the action of $p'$-groups on $p$-groups.

    In \Cref{sec:finitization} we give a pair of short arguments that immediately reduce \Cref{thm:main,thm:main-soluble} to the case of finite fields.
    The key tool here is the well-known theorem of Mal'cev asserting a strong form of residual finiteness for finitely generated linear groups.
    In the rest of the paper we take $\F$ to be finite.

    In \Cref{sec:trig} we prove some results about trigonalization of soluble subgroups and sections of $\GL_n(\F)$.
    A subgroup is called trigonalizable if it is conjugate to a group of upper-triangular matrices over some extension field.
    A section is called trigonalizable if it is the image of a trigonalizable subgroup.
    Another well-known theorem of Mal'cev asserts that soluble subgroups of $\GL_n(\F)$ are virtually trigonalizable.
    We prove a variant for soluble sections that moreover provides crucial control over the ``Weyl group'' of the section (\Cref{prop:trig+}).
    This control is essential for obtaining normality of $\Gamma$ in \Cref{thm:main}.

    The rest of the paper contains the body of the proof of \Cref{thm:main-soluble,thm:main}.
    \Cref{sec:general-to-soluble} covers the proof of \Cref{thm:main-soluble},
    while \Cref{sec:soluble-to-nilpotent} covers
    the deduction of \Cref{thm:main}.

    The starting points of \Cref{sec:general-to-soluble} (the general-to-soluble reduction) are the Product Theorem and Weisfeiler's structure theorem (\Cref{thm:weisfeiler}) for finite linear groups.
    Using these tools in combination with quasirandomness arguments and a few other facts about finite simple groups, we establish the following initial structure theorem, which can be viewed as an ``upside-down'' version of \Cref{thm:main-soluble}.

    \begin{theorem}
        \label{thm:tyukszem-intro}
        Let $\F$ be a finite field of characteristic $p > 0$.
        Let $A \subset \GL_n(\F)$ be a symmetric subset such that $|A^3| \leq K|A|$, where $K \geq 2$.
        Then there is a normal subgroup $\Gamma \nsgp \langle A \rangle$ such that
        \begin{enumerate}[(1)]
            \item $A$ is covered by $K^{O_n(1)}$ cosets of $\Gamma$,
            \item $\Gamma$ is soluble-by-$\Lie^*(p)$,
            \item $\Gamma / \Sol(\Gamma)$ is covered by $A^6$.
        \end{enumerate}
    \end{theorem}

    The main business of \Cref{sec:general-to-soluble} is to turn this structure ``right side up''.
    To do this we must show that the perfect core $P$ of $\Gamma$ is covered by $A^{O_n(1)}$ (quasirandomness upgrades this to $A^3$ automatically: see \Cref{lem:gowers-trick}).
    An argument based on the weak Ore conjecture shows that $A^{O_n(1)}$ covers $P / \Sol(P)$,
    and it follows from Weisfeiler's theorem that $N = [P, \Sol(P)]$ is a $p$-subgroup such that $[\Sol(P) : N] \leq O_n(1)$ (\Cref{cor:weisfeiler}),
    so the key is to show that $A^{O_n(1)}$ covers $N$.
    For this we use the following ``affine conjugating trick'' (\Cref{lem:affine-conjugating-trick}).

    \begin{lemma}[affine conjugating trick]
        \label{lem:affine-conjugating-trick-intro}
        Let $\Gamma = V \rtimes G$ be the semidirect product of an abelian group $V$
        and a $d$-generated $K^{21}$-quasirandom finite group $G$.
        Let $A \subset V$ be a symmetric $G$-invariant set generating $V$.
        If $|A^3| \leq K|A|$ then $A^{7d} \supset [V, G]$.
    \end{lemma}

    Finally, we show that $A^{O_n(1)}$ covers $N$ and hence $P$ using a rather tricky argument based on \Cref{lem:affine-conjugating-trick-intro} (see \Cref{prop:GN-prop}).
    This completes the proof of \Cref{thm:main-soluble}.

    Two key tools are used in \Cref{sec:soluble-to-nilpotent}, which proves \Cref{thm:main} starting from \Cref{thm:main-soluble}.
    The first is a powerful pivoting argument of Gill and Helfgott that can be seen as an extremely general version of a sum-product theorem: see \Cref{prop:pivoting}.
    The second is a new growth lemma for images of a bilinear map: see \Cref{prop:bilinear}.
    Additionally we make heavy use of the results on trigonalizable sections estasblished in \Cref{sec:trig}.

    The deduction of \Cref{thm:main} is divided into cases of increasing generality.
    To begin with we consider the case in which $G$ has a trigonalizable section $\Sigma = \Gamma / N$ and $A$ is a subset of $\Gamma$ whose image in $\Sigma$ has ``no small roots''.
    Here a \emph{root} is a homomorphism $\chi$ from $\Sigma$ to some abelian $p'$-group defined by the conjugation action of $\Sigma$ on a $\Sigma$-composition factor of $O_p(\Sigma)$:
    these are analogous to roots in the theory of Lie algebras.
    We say $A$ has \emph{no small roots} if $\chi(A)$ is either trivial or larger than $K^C$ for an appropriate constant $C$, for every root $\chi$ of $\Sigma$.
    In this favorable case we can show that $\gamma_n(\Sigma)$ is covered by $A^{O_n(1)}$.
    This argument is a somewhat involved induction on the nilpotency class of $\gamma_n(\Sigma)$, and uses a more general form of the idea of ``descent'' from \cite{gill-helfgott}. This argument is given in \Cref{sec:no-small-roots}.

    If, more generally, $A$ is not necessarily a subset of $\Gamma$ but still acts trivially on $\Sigma / O_p(\Sigma)$, then we use pigeonholing argument to shrink $\Sigma$ until the image of $A^2 \cap \Gamma$ in $\Sigma$ has no small roots, and then we apply the previous case to $A^4 \cap \Gamma$.
    By comparing $\Sigma_1 = \gen{A^2 \cap \Gamma}$ with $\Sigma_2 = \gen{A^4 \cap \Gamma}$ and using the no-small-roots property, we prove that $\gamma_n(\Sigma_1) = \gamma_n(\Sigma_2)$.
    This crucial conclusion allows us to identify an appropriate normal subgroup $\Delta \nsgp \langle A \rangle$ such that $P \leq \Delta \leq \Gamma$.
    This subgroup $\Delta$ fills the role of $\Gamma$ in \Cref{thm:main}.
    See \Cref{sec:good-section}.

    Finally, in general, the trigonalization theorem from \Cref{sec:trig} allows us to replace the soluble section provided by \Cref{thm:main-soluble} with a trigonalizable section $\Sigma$,
    and moreover guarantees that $\langle A \rangle$ has a bounded-index subgroup $G_0$ that acts trivially on $\Sigma / O_p(\Sigma)$.
    We complete the proof by applying the previous case to $A^{3m}
    \cap G_0$, where $m = [G:G_0]$, which is a generating set for $G_0$ by \Cref{lem:schreier}.

    The proof of \Cref{thm:polylog-dimeter-intro} is given in \Cref{sec:application}.
    In this section we recycle some of the same ideas.
    In particular the affine conjugating trick plays a key role.
    We also rely on an important result of Steinberg on
    representations of finite simple groups of Lie type.

    \subsection{Acknowledgments}

    BM would like to thank Harald Helfgott for introducing him to this
    problem, as well as Vlad Finkelstein, Jonathan Pakianathan, Lam
    Pham, Misha Rudnev, Matthew Tointon, and James Wheeler for helpful
    conversations.

    \section{Notation}
    \label{sec:notation}

    The commutator $[x, y]$ is defined as $x^{-1} y^{-1} x y$.
    Iterated commutators are defined by $[x, \dots, y, z] = [[x, \dots, y], z]$.
    If $H, K$ are subgroups then $[H, K] = \langle [h, k] : h \in H, k \in K\rangle$,
    and similarly for iterated commutators.
    The subgroup $[H, K]$ is a normal subgroup of $\langle H, K\rangle$ (\cite{aschbacher}*{(8.5.6)}).
    If $H$ and $K$ are groups and $H$ acts on $K$ then we may always consider $H$ and $K$ as subgroups of their semidirect product $K \rtimes H$.
    The commutator $[H, K] \leq K$ and centralizers $C_H(K)$ and $C_K(H)$ are defined accordingly.
    As usual $G' = [G, G]$.
    The derived series of $G$ is denoted $(G^{(n)})_{n \geq 0}$,
    and $G^{(\omega)} = \bigcap_{n \ge 0} G^{(n)}$.
    If the derived series of $G$ terminates in finitely many steps (e.g., if $G$ is finite), then $G^{(\omega)}$ is called the \emph{perfect core} of $G$.
    The lower central series is denoted $(\gamma_n(G))_{n \geq 1}$,
    and $\gamma_\omega(G) = \bigcap_{n\ge 1} \gamma_n(G)$.

    A finite group is a \emph{$p$-group} if its order is a power of $p$;
    it is a \emph{$p'$-group} if its order is prime to $p$.
    The largest normal $p$-subgroup of a finite group $G$ is denoted $O_p(G)$ and called the \emph{$p$-core}.
    The largest normal soluble subgroup is denoted $\Sol(G)$ and called the \emph{soluble radical}.

    If $G$ is a group, $\frat(G)$ denotes the Frattini subgroup.
    If $G$ is finite, $\frat(G)$ is nilpotent, by the Frattini argument.
    If $G$ is a $p$-group, the Burnside basis theorem asserts that $\frat(G)$ is the smallest normal subgroup such that $G / \frat(G)$ is elementary abelian (see \cite{aschbacher}*{(23.2)}).

    All of the subgroups just defined are characteristic subgroups.
    In general if $H$ is a characteristic subgroup of $G$ we write $H \csgp G$.

    A \emph{section} of a group $G$ is a quotient $\Sigma = H / N$ where $N \nsgp H \leq G$.
    The \emph{normalizer} of $\Sigma$ is $N_G(\Sigma) = N_G(H) \cap N_G(N)$.
    Note that the normalizer of $\Sigma$ acts naturally on $\Sigma$.
    The \emph{centralizer} of $\Sigma$ is $C_G(\Sigma) = \{g \in N_G(\Sigma) : g~\text{acts trivially on}~\Sigma\}$.
    If $A \subset G$, the \emph{trace} of $A$ in $\Sigma$, denoted $\tr(A, \Sigma)$, is the projection of $A \cap H$ to $\Sigma$.
    We say $A$ \emph{covers} $\Sigma$ if $\tr(A, \Sigma) = \Sigma$, or equivalently if $H \subset AN$.

    Throughout $\F$ denotes a field, usually finite.
    The (Borel) subgroup of elements represented by an upper-triangular matrix in the standard basis is denoted $B_n(\F)$.
    The (toral) subgroup of diagonal matrices is denoted $T_n(\F)$, while the subgroup of upper unitriangular matrices is denoted $U_n(\F)$.
    Note that $B_n(\F) = U_n(\F) T_n(\F) \cong U_n(\F) \rtimes T_n(\F)$.
    There is a projection map $\pi : B_n(\F) \to T_n(\F)$ such that $\ker\pi = U_n(\F)$.
    Various other groups of the form $G = U \rtimes T$ will arise and usually we denote by $\pi$ the projection map $G \to T$ such that $\ker \pi = U$.
    Maps such as $\pi$ are denoted interchangeably on the left or exponentially, so $\pi(A) = A^\pi$.

    \section{Examples}
    \label{sec:examples}

    Well-known basic examples such as subgroups and nilprogressions demonstrate the necessity of the basic elements in the main theorem \Cref{thm:main}.
    See for example \cite{breuillard-intro}*{Section~1.6}.
    In this section we give some examples which demonstrate the necessity of some subtler aspects of \Cref{thm:main}.

    \subsection{Failure of polynomiality in high rank}

    Unlike \Cref{thm:main}, the bound in \cite{BGT} for the number of cosets of $\Gamma$ required to cover $A$ cannot be polynomial in $K$,
    even if we only require $\Gamma$ to be finite-by-soluble.
    The following example appeared in the unpublished manuscript \cite{eberhard-arxiv}, and was based on a similar construction in \cite{breuillard-tointon}*{Section~4.1}.

    Let $G = \Z^n \rtimes S_n$. and let $A = A_r = [-r, r]^n S_n$ for any $r > n!$.
    It is easy to see that $A$ is a $2^n$-approximate group.
    Suppose $\Gamma \le G$ is a subgroup such that $m = |A \Gamma / \Gamma| < n!$.
    Since $A_r = A_1^r$, it follows that $|A_1^{s+1} \Gamma / \Gamma| = |A_1^s \Gamma / \Gamma|$ for some $s < r$, which implies that $G = A \Gamma$ since $A$ generates $G$.
    Hence $\Gamma$ has index $m$ in $G$.
    In particular
    \begin{enumerate}[(1)]
        \item $\Gamma \cap S_n$ has index at most $m$ in $S_n$,
        \item $\Gamma \cap \Z^n$ contains $m\Z^n$.
    \end{enumerate}

    Note that $\Gamma$ has no nontrivial finite normal subgroup.
    Indeed if $H \nsgp \Gamma$ is finite then $H \cap \Z^n = 1$
    and $[m\Z^n, H] \leq \Z^n \cap H = 1$, so $H$ acts trivially on $\Z^n$, so $H \leq \Z^n$, so $H = 1$.

    If $\Gamma$ is soluble then so is $\Gamma \cap S_n$, which implies $m \geq n! / 24^{(n-1)/3}$ by a result of Dixon~\cite{dixon-solvable}.
    If $\Gamma$ is in fact nilpotent then $\Gamma \cap S_n$ must be trivial, so $m \geq n!$.

    Note that $A$ is a $2^n$-approximate group, and $G$ is isomorphic to a subgroup of $\GL_{n+1}(\Z)$.
    This shows that the number of cosets required in
    \Cref{BGT} must be at least $K^{c \log \log K}$,
    and that in \Cref{thm:main} or \Cref{thm:main-soluble} must be at least $K^{c \log n}$.

    \subsection{Failure of normality in high rank}

    Normality of $\Gamma$ in $G$, as stated in \Cref{thm:main} but not in \Cref{BGT}, is a special feature of bounded rank.
    The following example is due to Pyber. It was previously mentioned in \cite{breuillard-intro}*{Example~1.17} and \cite{helfgott--perspectives}*{p.~53}).

    Let $n$ be odd and let $G = S_n$ be the symmetric group of degree $n$.
    Let $m = \floor{n/4}$ and let $\Gamma_0$ be the elementary abelian subgroup
    \[
        \Gamma_0 = \langle (1,2), (3, 4), \dots, (2m-1, 2m) \rangle \cong C_2^m.
    \]
    Let $\sigma$ be the $n$-cycle $x \mapsto x + 2$.
    Let $A = \Gamma_0 \cup \{\sigma^{\pm1}\}$.
    Then $A$ is a $10$-approximate group which generates $G$,
    and therefore subject to \Cref{BGT}.
    However, there are only three normal subgroups of $G$, and none of them is suitable.

    Many variants of this construction are possible. For example we can take $\Gamma_0$ to be a direct product of copies of $A_5$.

    \subsection{Normality vs finite generation}

    Even in bounded rank we cannot guarantee normality simultaneously with finite generation of $\Gamma$ (so in particular the rank of $\Gamma / H$ in \Cref{thm:main} cannot be controlled).
    This example also illustrates some of the challenges of non-prime fields.

    Let $\F = \F_p(t)$, where $p = 2$ (say) and $t$ is transcendental over $\F_p$.
    Let $V = \F_p[t, t^{-1}]$ and let
    \[
        G = \left\{ \begin{pmatrix} t^n & v \\ 0 & 1 \end{pmatrix} : n \in \Z, v \in V \right\} \cong \F_p^\Z \rtimes \Z.
    \]
    Note that $G \leq \GL_2(\F)$.
    Let $W \leq V$ be the linear subspace spanned by $t^e$ for $-d \leq e \leq d$ and let
    \[
        A =
        \begin{pmatrix} 1 & W \\ 0 & 1 \end{pmatrix}
        \cup
        \left\{ \begin{pmatrix} t & 0 \\ 0 & 1 \end{pmatrix}^{\pm 1} \right\}.
    \]
    Then $A$ is a $(2p+5)$-approximate group which generates $G$,
    but $G$ has no finitely generated normal finite-by-nilpotent subgroup apart from the trivial group.

    \subsection{Lack of other structure in high rank}

    We can give diverse examples of large generating subsets of small
    growth in high-rank simple groups.
    The main issue is to ensure that these subsets actually generate
    our group, which is a crucial condition in the Product Theorem for
    simple groups of bounded rank.
    We use
    the following trick inspired by an idea of Bannai~\cite{bannai}.
    Let $A = S \cup C$, where $S$ is an arbitrary $K$-tripling
    subset of some group $G$ and $C$ is a conjugacy class of $G$.
    Then $A^3 \subset S^3 \cup S^2 C \cup S C^2 \cup C^3$,
    so $A$ is $K'$-tripling for $K' = K + K|C| + 2|C|^2$.
    If $|C|$ is not much larger than $K$, this shows that the growth
    of $A$ is similar to that of $S$.
    When $G$ is a finite simple group then certainly $C$ generates $G$,
     and this is also often the case when $G$ is only almost simple.

    For example, let $m \leq n$ and let $\Gamma$ be a copy of $S_m$ in $S_n$ (or similarly a copy of $\SL_m(\F_2)$ in $\SL_n(\F_2)$) and let $C$ be a small conjugacy class, say the set of transpositions (or transvections in $\SL_n(\F_2)$). Taking $m = n/2$ or even $m = n-100$, we get a huge generating set $A = \Gamma \cup C$ of relatively small growth.

    For another example let $\Gamma$ be a cyclic subgroup of $G = S_n$ of order $N \approx \exp(c \sqrt{n \log n})$, let $S$ be an interval of length $\sqrt{N}$ in $\Gamma$, and let $C_0$ be a set of $n-1$ transpositions that generate $S_n$.
    Then $A = S \cup C_0$ is a $O(n^6)$-tripling generating set of size $\exp(c \sqrt{n \log n})$, but $A$ is far from containing a subgroup and not dense in any conjugacy class.

    Let us note that, by a result of Guralnick and Saxl~\cite{guralnick-saxl}, if $G$ is almost simple then we can choose a small subset $C_0 \subset C$ which generates $\langle C \rangle$,
    so $A=S\cup C_0$ generates $\langle S, C \rangle$.
    If $S$ is a $K$-tripling subset of $G$ then $A$ is a $K'$-tripling
    subset for $K'=K+K|C|+|C|^3$, but for a suitable choice of $S$
    the generating set $A$ is far from being dense in any conjugacy class.

    These examples show in particular that the Product Theorem fails
    completely for finite simple groups of unbounded rank.

    \section{Toolbox}
    \label{sec:toolbox}

    \subsection{Basic arithmetic combinatorics}

        We review some basic theory of arithmetic combinatorics and in particular sets of small tripling.
        Most of what we need is in \cite{helfgott--perspectives}*{Sections~3 and 4}.
        Throughout we will work exclusively with sets of small tripling, so we will not review any results or terms related strictly to approximate groups.

        For $A, B$ subsets of a group $G$, we write $AB$ for the product set
        \[
            AB = \{ab : a \in A, b \in B\}.
        \]
        Define $A^k$ for $k \geq 0$ to be the set of all products $a_1 \cdots a_k$ with $a_1, \dots, a_k \in A$.
        We also define $A^{-1} = \{a^{-1} : a \in A\}$ and $A^{-k} = (A^{-1})^k$.
        We write $A^{\pm k}$ for $(A \cup \{1\} \cup A^{-1})^k$.
        We call $A$ \emph{symmetric} if $A = A^{-1}$.
%

        \begin{lemma}[tripling lemma, see \cite{helfgott--perspectives}*{(3.3)}]
            \label{lem:tripling}
            Let $A \subset G$ be finite, nonempty, and symmetric.
            For $k \geq 3$,
            \[
                |A^k| / |A| \leq (|A^3| / |A|)^{k-2}.
            \]
        \end{lemma}

        \begin{lemma}[orbit--stabilizer for sets, see \cite{helfgott--perspectives}*{Lemma~4.1}]
            \label{lem:orbit-stab}
            Let $A, B \subset G$ be finite sets and $H \leq G$ a (not necessarily normal) subgroup.
            Let $\pi : G \to G/H$ be the quotient map.
            \begin{enumerate}[(1)]
                 \item $|A^\pi| |B \cap H| \leq |AB|$.
                 \item $|A| \leq |A^\pi| |A^{-1}A \cap H|$.
             \end{enumerate}
        \end{lemma}

        Note that if $A = B$ is a subgroup then the lemma states $|A| = |A^\pi| |A\cap H|$, which is the orbit--stabilizer theorem for the action of $A$ on $G/H$.
        The lemma implies that there is in general still some relation between the ``orbit'' $A^\pi$ and the ``stabilizer'' $A\cap H$.
%

        \begin{lemma}[growth in subgroups, quotients, and sections]
            \label{lem:growth-in-sections}
            Let $A \subset G$ be nonempty, finite, symmetric, and $K$-tripling.
            \begin{enumerate}[(1)]
                \item For $H \leq G$,
                \[
                    \frac{|A^k \cap H|}{|A^2 \cap H|}
                    \leq \frac{|A^{k+1}|}{|A|}
                    \leq K^{k-1}.
                \]
                \item For $H \leq G$ and $\pi : G \to G/H$ the quotient map,
                \[
                    \frac{|(A^k)^\pi|}{|A^\pi|}
                    \leq \frac{|A^{k+2}|}{|A|}
                    \leq K^k.
                \]
                \item For $\Sigma = H/N$ a section of $G$,
                \[
                    \frac{|\tr(A^k, \Sigma)|}{|\tr(A^2, \Sigma)|}
                    \leq \frac{|A^{k+5}|}{|A|}
                    \leq K^{k+3}.
                \]
            \end{enumerate}
        \end{lemma}
        \begin{proof}
            These all follow from \Cref{lem:tripling,lem:orbit-stab}.
            See \cite{helfgott--perspectives}*{Section~4} for the first two.
            The third is similar.
            Suppose $N \nsgp H \leq G$ and let $\pi : H \to H / N$ be the quotient map.
            Applying \Cref{lem:orbit-stab}(1) to $(A^k \cap H, A^4 \cap H)$
            we obtain
            \[
              |(A^k\cap H)^\pi| |A^4\cap N| \le
              |(A^k\cap H)(A^4\cap H)| \le
              |A^{k+4}\cap H|,
            \]
            and applying \Cref{lem:orbit-stab}(2) to $A^2 \cap H$
            we obtain
            \[
              |A^2\cap H| \le
              |(A^2\cap H)^\pi| |(A^{-2}\cap H)(A^2\cap H)\cap N| \le
              |(A^2\cap H)^\pi| |A^4\cap N|.
            \]
            Multiplying these inequalities we obtain
            \[
                \frac{|(A^k \cap H)^\pi|}{|(A^2 \cap H)^\pi|}
                \leq
                \frac{|A^{k+4} \cap H|}
                {|A^2 \cap H|}.
            \]
            Now applying (1) gives (3).
        \end{proof}

        \begin{lemma}[covering lemma, essentially \cite{helfgott--perspectives}*{Lemma~4.2}]
            \label{lem:covering-observation}
            Suppose $A \subset G$ is finite and covered by $k$ left cosets of $H \leq G$.
            Then $A$ is covered by $k$ left translates of $A^{-1} A \cap H$.
        \end{lemma}
        \begin{proof}
            Suppose $A \subset \bigcup_{i=1}^k x_i H$.
            We may assume $x_1, \dots, x_k \in A$.
            Then $A \cap x_iH = x_i(x_i^{-1}A \cap H) \subset x_i (A^{-1}A \cap H)$.
            Hence $A \subset \bigcup_{i=1}^k x_i (A^{-1}A \cap H)$.
        \end{proof}

        The follow basic and self-evident rule will come up several times.

        \begin{lemma}[modular law]
            Let $A, B \subset G$ and assume $B \subset \Gamma \leq G$. Then
            \[
                \Gamma \cap AB = (\Gamma \cap A)B.
            \]
            More generally, for $A, B, C \subset G$,
            \[
                C \cap AB \subset (CB^{-1} \cap A) B \subset CB^{-1} B \cap AB.
            \]
        \end{lemma}

        We need the following version of Schreier's lemma
        (see \cite{Helfgott-Seress}*{Lemma~3.8}).

        \begin{lemma}[Schreier]
            \label{lem:schreier}
            Let $G$ be a group and $A \subset G$ and $H \leq G$.
            Suppose $AH = G$.
            Then $\langle A\rangle \cap H = \langle A^{\pm 3} \cap H\rangle$.
        \end{lemma}

        This lemma is best known in the context of finitely generated groups,
        as it implies that a finite-index subgroup of a finitely generated group is finitely generated.
        The elements of $A^{\pm3} \cap H$ are sometimes called Schreier generators for $H$.




        \subsection{Quasirandomness}

        Next we recall some facts about quasirandomness.
        If $G$ is a finite group let $\deg_\C(G)$ be the minimum degree of a nontrivial complex representation of $G$,
        conventionally $\infty$ if $G$ is trivial.
        A group is colloquially called \emph{quasirandom} if $\deg_\C(G)$ is large.
        Simple groups of Lie type are examples.

        \begin{theorem}[Landazuri--Seitz~\cite{landazuri-seitz}]
            \label{landazuri--seitz}
            If $L$ is a simple group of Lie type of characteristic $p$ and rank $\ell$ then $\deg_\C(L) \geq |L|^{c/\ell}$ for a constant $c>0$.
        \end{theorem}

        The relevance of quasirandomness is indicated by the following well-known result, essentially due to Gowers (see also \cites{gowers,nikolov-pyber}).

        \begin{proposition}[\cite{babai-nikolov-pyber}*{Corollary~2.6}]
            \label{prop:quasirandomness}
            Let $G$ be a finite group and $m = \deg_\C(G)$.
            If $A, B, C \subset G$ are sets such that
            \[
                |A| |B| |C| \ge |G|^3 / m,
            \]
            then
            \[
                ABC = G.
            \]
            In particular if $|A| \ge |G| / m^{1/3}$ then $A^3 = G$.
        \end{proposition}

        \begin{lemma}
            \label{lem:gowers-trick}
            Let $A$ be a symmetric subset of a finite group $G$ such that $A^k$ contains a coset of $N \nsgp G$.
            If $|A^3| \leq K|A|$ and $\deg_\C(N) \geq K^{3k}$ then $A^3$ contains a coset of $N$.
            In particular if $N = G$ then $A^3 = G$.
        \end{lemma}
        \begin{proof}
            We may assume $k \geq 3$ and $K > 1$.
            Since $A^k$ contains a coset of $N$, $|A^{k+1}| \geq |N| |A^\pi|$, where $\pi : G \to G/N$ is the projection.
            Hence there is some $x \in G$ such that
            \[
                |A \cap xN| \geq |A| / |A^\pi| \geq (|A| / |A^{k+1}|) |N| \geq K^{-k+1} |N|
            \]
            by \Cref{lem:tripling}.
            Let $B = x^{-1}A \cap N$, so $|B| \geq K^{-k+1} |N|$.
            Then
            \[
                A^3 \supset xBxBxB = x^3 B^{x^2} B^x B
            \]
            and, by \Cref{prop:quasirandomness}, $B^{x^2} B^x B = N$.
        \end{proof}

        \begin{lemma}
            \label{lem:degrees}
            Let $G$ be a finite perfect group. Then
            \[
                \deg_\C(G) \geq c \deg_\C(G/\Sol(G))^{1/2}.
            \]
        \end{lemma}
        \begin{proof}
            Consider a nontrivial complex representation $G^\pi$ of $G$ of degree $d$.
            Since $G$ is perfect, $G^\pi$ is not soluble.
            By \cite{nikolov-pyber}*{Corollary~2.5}, $G^\pi / \Sol(G^\pi)$ embeds into $\Sym(d_1)$ for some $d_1 \leq C d^2$.
            Since $\Sol(G)^\pi \leq \Sol(G^\pi)$, $G / \Sol(G)$ has a nontrivial permutation representation of degree $d_1$.
            Hence
            \[
                \deg_\C(G / \Sol(G)) \leq C d^2.\qedhere
            \]
        \end{proof}

        \subsection{Coprime action}

        \begin{theorem}[Schur--Zassenhaus, see \cite{aschbacher}*{(18.1)}]
            \label{thm:SZ}
            Let $G$ be a finite group.
            Assume $U \nsgp G$ and $\gcd(|U|, |G/U|) = 1$.
            Assume $U$ or $G/U$ is soluble.
            Then there is a complement $T$ to $U$ in $G$
            and all complements are conjugate.
        \end{theorem}

        For the rest of this subsection assume $U$ is a finite $p$-group and $T$ is a finite $p'$-group acting on $U$.
        Let $G = U \rtimes T$.
        The following lemma generalizes a familiar property of vector spaces.

        \begin{lemma}[\cite{aschbacher}*{(24.4--6)}]
            \label{lem:U-decomp}
            $U = [U, T] C_U(T)$ and $[U, T] = [U, T, T]$.
            If $U$ is abelian then $U = [U, T] \times C_U(T)$,
            and in particular $[U, T] = U$ if and only if $C_U(T) = 1$.
        \end{lemma}
%
%

        In general the intersection $[U, T] \cap C_U(T)$ is nontrivial.
        For example, let
        \[
            G = \left\{ \begin{pmatrix}
            1 & x & z \\
            0 & a & y \\
            0 & 0 & 1 \\
            \end{pmatrix}
            : a \in \F_p^\times, x, y, z \in \F_p
            \right\} = TU,
        \]
        where $T = G \cap T_3(\F_p)$ and $U = U_3(\F_p)$.
        Then $[U, T] = U$ and $C_U(T) = Z(U) \cong C_p$.
        However, the intersection $[U, T] \cap C_U(T)$ is always contained in the commutator subgroup of $[U, T]$.

        \begin{lemma}
            \label{lem:CHT}
            Let $H = [U, T]$. Then $C_{H / H'}(T) = 1$.
        \end{lemma}
        \begin{proof}
            By \Cref{lem:U-decomp}, $H = [H, T]$.
            It follows that $\bar H = [\bar H, T]$ where $\bar H = H / H'$.
            On the other hand, by \Cref{lem:U-decomp}, $\bar H = [\bar H, T] \times C_{\bar H}(T)$.
            Hence $C_{\bar H}(T) = 1$.
        \end{proof}

        \begin{lemma}
            \label{lem:gamma_omega}
            Assume $\gamma_n(U) = \gamma_n(T) = 1$.
            Then $\gamma_\omega(G) = \gamma_n(G) = [U, T]$.
        \end{lemma}
        \begin{proof}
            By \Cref{lem:U-decomp}, $[U, T] = [U, T, T]$. It follows that $[U, T]  \leq \gamma_\omega(G)$.
            On the other hand $G / [U,T] \cong U / [U, T] \times T$
            is nilpotent of class at most $n$, so $[U, T] \geq \gamma_n(G)$.
        \end{proof}

    \section{Finitization}
    \label{sec:finitization}

    The following well-known theorem of Mal'cev asserts that finitely generated linear groups are residually finite.
    More strongly, it asserts that we can distinguish the elements of any finite subset using a finite residue field.
    We will use it to reduce \Cref{thm:main} and \Cref{thm:main-soluble} to the finite field case.

    \begin{theorem}[Mal'cev, see \cite{wehrfritz}*{Theorem~4.2}]
        \label{thm:malcev-locally-residually-finite}
        Let $\F$ be a field, $G \leq \GL_n(\F)$ a finitely generated subgroup,
        and $S \subset G$ a finite subset.
        Then there is a finite field $\K$ and a homomorphism $\pi : G \to \GL_n(\K)$ such that $\pi$ is injective on $S$.
        If $\chr \F > 0$ then $\chr \K  =\chr \F$.
    \end{theorem}

    \begin{proposition}
        \label{prop:main-finitization}
        If \Cref{thm:main} holds for finite fields then it holds in general.
    \end{proposition}
    \begin{proof}
    Let $\F$ be an arbitrary field and let $A \subset \GL_n(\F)$ be finite, nonempty, symmetric, and $K$-tripling.
    Let $G = \langle A \rangle$.
    By \Cref{thm:malcev-locally-residually-finite}, for any finite set $S \subset \GL_n(\F)$ there is a finite field $\K$ and a homomorphism $\pi : \gen {A \cup S} \to \GL_n(\K)$ that is injective on $S$.
    Let $p = \chr K$.
    Note that $\gen {\pi(A)} = \pi(G)$.
    By \Cref{lem:growth-in-sections}(2), $\pi(A)$ is $K^3$-tripling.
    Assuming \Cref{thm:main} holds in the finite field case, there is $m \le O_n(1)$ and $\Gamma_0 \nsgp \pi(G)$ such that
    \begin{enumerate}[(i)]
        \item $\pi(A)$ is covered by $K^m$ cosets of $\Gamma_0$,
        \item $H_0 = \gamma_n(\Gamma_0)$ is contained in $\pi(A)^m$,
        \item $H_0$ has a perfect soluble-by-$\Lie^*(p)$ normal subgroup $P_0$ contained in a translate of $\pi(A)^3$ and $H_0/P_0$ is a $p$-group.
    \end{enumerate}
    Let $\Gamma = \pi^{-1}(\Gamma_0) \cap G$. Then $\Gamma \nsgp G$, $\pi(\Gamma) = \Gamma_0$, and $\pi(\gamma_n(\Gamma)) = \gamma_n(\Gamma_0)$. We thus obtain the following properties:
    \begin{enumerate}[(1)]
        \item $A$ is covered by $K^m$ cosets of $\Gamma$,
        \item $H_0 = \pi(\gamma_n(\Gamma))$ is contained in $\pi(A^m)$,
        \item $H_0$ has a perfect soluble-by-$\Lie^*(p)$ normal subgroup $P_0$ contained in a translate of $\pi(A)^3$ and $H_0/P_0$ is a $p$-group.
    \end{enumerate}
    Moreover, by \Cref{lem:covering-observation}, $A$ is covered by $K^m$ translates of $A^2 \cap \Gamma$,
    and $P_0 \subset \pi(A)^6$,
    so we retain properties (1)--(3) if we replace $\Gamma$ with the subgroup $\gen {(A^6 \cap \Gamma)^G}$ normally generated by $A^6 \cap \Gamma$.
    Thus we may append the following property:
    \begin{enumerate}[(1), resume]
        \item $\Gamma = \gen {(A^6 \cap \Gamma)^G}$.
    \end{enumerate}

    At the moment, $\Gamma$ depends on $S$, but we can eliminate this dependence as follows.
    Call $\Gamma$ \emph{good for $S$} if there is a finite field $\K$ (with $\chr K = \chr F$ if $\chr F > 0$) and some homomorphism $\pi : \gen {A \cup S} \to \GL_n(\K)$ injective on $S$
    such that (1)--(4) hold for $\Gamma$.
    By the argument above, for every finite set $S \subset \GL_n(\F)$ there is some subgroup $\Gamma \nsgp G$ that is good for $S$.
    Moreover, if $S_1 \subset S_2$ and $\Gamma$ is good for $S_2$ then $\Gamma$ is good for $S_1$.
    Since $A^6$ is finite, property (4) implies that there are only finitely many possibilities for $\Gamma$, say $\Gamma_1, \dots, \Gamma_N$.
    Suppose, for each $i$, $\Gamma_i$ is not good for some finite set $S_i$.
    Then none of $\Gamma_1, \dots, \Gamma_N$ is good for the finite set $S = S_1 \cup \cdots \cup S_N$, which is a contradiction.
    Thus we may assume that $\Gamma$ is independent of $S$.

    Now let $H = \gamma_n(\Gamma)$. If $x \in H$ then by applying (2) above with $S = A^m \cup \{x\}$ we obtain $x \in A^m$.
    Thus $H \subset A^m$.
    Moreover, since we may assume $\pi$ is injective on $A^m$, $\pi$ restricts to an isomorphism of $H = \gamma_n(\Gamma)$ with $H_0 = \pi(\gamma_n(\Gamma))$.
    Hence $H$ has a perfect soluble-by-$\Lie^*(p)$ normal subgroup $P \cong P_0$ such that $H / P$ is a $p$-group, where $p = \chr \K > 0$.
    Also, $x_0 P_0 \subset \pi(A)^3$ for some $x_0 \in \pi(G)$.
    In particular $x_0 \in \pi(A)^3 = \pi(A^3)$, so $x_0 = \pi(x)$ for some $x \in A^3$ and $\pi(x P) \subset \pi(A^3)$.
    Since we may assume $\pi$ is injective on $A^3 P$, this implies $x P \subset A^3$. Thus $P$ is contained in a translate of $A^3$.

    If $\chr \F > 0$ then $p = \chr \F$ and we are done. If $\chr \F = 0$ then by choosing $S$ to include many transvections $1 + x e_{12}$ ($x \in \Z$) we can force $p$ to be larger than $|H|$, which implies that $H$ is trivial.
    \end{proof}

    We can reduce \Cref{thm:main-soluble} to the finite field case using almost the same argument.

    \begin{proposition}
        \label{prop:main-soluble-finitization}
        If \Cref{thm:main-soluble} holds for finite fields then it holds in general.
    \end{proposition}
    \begin{proof}
        We argue exactly as in the proof of \Cref{prop:main-finitization}.
        We obtain $\Gamma \nsgp G$ and $m = O(\log n)$ such that, for any finite subset $S \subset \GL_n(\F)$, there is a finite field $\K$ (with $\chr K = \chr F$ if $\chr F > 0$) and a homomorphism $\pi : \gen{A \cup S} \to \GL_n(\K)$ injective on $S$ such that
        \begin{enumerate}[(1)]
            \item $A$ is covered by $K^{O_n(1)}$ cosets of $\Gamma$
            \item $P_0 = \pi(\Gamma)^{(m)}$ is perfect, soluble-by-$\Lie^*(p)$, where $p = \chr K$, and contained in a translate of $\pi(A)^3$.
        \end{enumerate}
        Let $P = \Gamma^{(m)}$ and note $\pi(P) = P_0$.
        By (2), there is some $x_0 \in \pi(G)$ such that $x_0 P_0 \subset \pi(A)^3$. In particular $x_0 \in \pi(A)^3 = \pi(A^3)$, so $x_0 = \pi(x)$ for some $x \in A^3$, and $\pi(x P) \subset \pi(A^3)$.
        A priori $x$ depends on $S$, but since $x \in A^3$ and $A^3$ is finite we can eliminate this dependence as in the previous proof.
        Now if if $y \in P$ then by taking $S = A^3 \cup \{xy\}$ it follows that $xy \in A^3$.
        Thus $x P \subset A^3$.
        In particular $P$ is finite, and taking $S \supset P$ gives $P \cong \pi(P) = P_0$.
        Thus $P$ is perfect, soluble-by-$\Lie^*(p)$, and contained in a translate of $A^3$.
        If $\chr F > 0$ then $p = \chr F$ and we are done, and otherwise we may choose $S$ to force $p$ to be larger than $|P|$, which shows that $P$ must be trivial.
    \end{proof}

    \section{Trigonalization}
    \label{sec:trig}

    In this section and the next we extensively use arguments from the theory of linear groups, particularly soluble linear groups. A nice general reference on linear groups is \cite{wehrfritz}.

    Call a subgroup of $\GL_n(\F)$ \emph{trigonal} if it is contained in $B_n(\F)$,
    and \emph{trigonalizable} if it is conjugate to a subgroup of $B_n(\E)$ for some extension $\E$ of $\F$.
    Virtually soluble is equivalent to virtually trigonalizable, by another well-known theorem of Mal'cev.

    \begin{theorem}[Mal'cev, see \cite{wehrfritz}*{Theorem~3.6}]
        \label{thm:trigonalization}
        Every soluble subgroup $G$ of $\GL_n(\F)$
        has a trigonalizable normal subgroup $G_0$ of index $O_n(1)$.
    \end{theorem}

    It follows immediately from \Cref{thm:trigonalization} that any soluble subgroup of $\GL_n(\F)$ has derived length $O_n(1)$,
    a result originally proved by Zassenhaus.
    This bound was sharpened by Newman to $O(\log n)$: see \cite{wehrfritz}*{Chapter 3}.
    The following lemma generalizes this bound to virtually soluble subgroups.

    \begin{lemma}
        \label{lem:P}
        Let $\Gamma$ be a finite subgroup of $\GL_n(\F)$.
        \begin{enumerate}[(1)]
            \item For any $m \geq 0$ there is a soluble subgroup $S \leq \Gamma$ such that $\Gamma = \Gamma^{(m)}S$.
            \item There is $m = O(\log n)$ such that $\Gamma^{(m)}$ is perfect.
        \end{enumerate}
    \end{lemma}
    \begin{proof}
        (1)
        Let $P = \Gamma^{(m)}$.
        Let $S$ be a minimal subgroup of $\Gamma$ such that $\Gamma = PS$.
        By minimality of $S$, if $M < S$ then $\Gamma > PM = P(P \cap S) M$,
        so $(P \cap S) M < S$.
        Hence $P \cap S \leq \frat(S)$.
        This implies that $P \cap S$ is nilpotent.
        On the other hand $S / (P \cap S) \cong PS / P = \Gamma / P$ is soluble.
        Hence $S$ is soluble.

        (2) For soluble $\Gamma$ this was proved by Newman, as mentioned.
        In general let $m$ be large enough for the soluble case and let $\Gamma$ be arbitrary.
        Let $P = \Gamma^{(m)}$.
        By (1), there is soluble $S \leq \Gamma$ such that $\Gamma = P'S$.
        Then $P = \Gamma^{(m)} \leq P' S^{(m)} = P'$, so $P$ is perfect.
    \end{proof}

    We say a group $G$ is \emph{$p$-by-abelian} if $G / O_p(G)$ is abelian;
    if $G$ is finite then by Schur--Zassenhaus (\Cref{thm:SZ}) this is equivalent to a semidirect decomposition $G = O_p(G) T$ for some abelian $p'$-group $T \leq G$.

    \begin{lemma} \label{trigonalizable-subgroups-are-PT}
        Let $\F$ be a field of characteristic $p > 0$.
        A finite subgroup $G \leq \GL_n(\F)$ is trigonalizable if and only if it is $p$-by-abelian.
        In this case $G / O_p(G)$ is the direct product of at most $n$ cyclic groups and $\gamma_n(O_p(G)) = 1$.
    \end{lemma}
    \begin{proof}
        We may assume $\F$ is algebraically closed.
        If $G$ is trigonalizable then without loss of generality $G \leq B_n(\F)$,
        so $O_p(G) = G \cap U_n(\F)$
        and $G / O_p(G) \cong GU_n(\F) / U_n(\F)$ is isomorphic to a subgroup of $T_n(\F) \cong (\F^\times)^n$, hence a direct product of at most $n$ cyclic groups.

        Conversely suppose $G$ is $p$-by-abelian.
        Let $U = O_p(G)$.
        Let $V$ be a $G$-composition factor of $\F^n$.
        Then $V$ is an irreducible $\F G$-module, so by Clifford's theorem it is a completely reducible $\F U$-module,
        but the only irreducible $\F U$-module is the trivial one-dimensional module, so $U$ acts trivially on $V$.
        Hence $V$ is an irreducible $\F(G / U)$-module, which implies that $\dim V = 1$ since $G / U$ is abelian and $\F$ is algebraically closed.
    \end{proof}

    For $N \nsgp \Gamma \leq \GL_n(\F)$, we say $\Sigma = \Gamma / N$ is a \emph{trigonalizable section} if $\Gamma = BN$ for some trigonalizable subgroup $B \le \GL_n(\F)$.
    By the previous lemma, such a subgroup $B$ must be $p$-by-abelian, so $\Sigma$ must be $p$-by-abelian.
    The converse does not hold. For example, if $\Gamma$ is a nonabelian nilpotent $p'$-group contained in $\GL_n(\F_p)$ then the section $\Sigma = \Gamma / \Gamma'$ is abelian but not covered by a ($p$-by-)abelian subgroup of $\Gamma$.

    The following key proposition establishes a variant of Mal'cev's theorem for soluble sections.
    Moreover we have control over the ``Weyl group'' of the trigonalizable subsection.

    \begin{theorem}
        \label{prop:trig+}
        Let $\F$ be a finite field of characteristic $p$.
        Let $\Sigma = \Gamma / N$ be a soluble section of $L = \GL_n(\F)$.
        Then there is a trigonalizable subsection $\Sigma_0 \csgp \Sigma$ of index $O_n(1)$ containing $O_p(\Sigma)$.
        Moreover, if $R = N_L(\Sigma) = N_L(\Gamma) \cap N_L(N)$ is the section normalizer then $[R : C_R(\Sigma_0 / O_p(\Sigma))] \leq n!$.
    \end{theorem}

    \begin{proof}
        We will establish the following properties in order (for $\Sigma_0 \csgp \Sigma$ of index $O_n(1)$ to be determined):
        \begin{enumerate}[(1)]
            \item\label{item1} $\Sigma_0$ is $p$-by-abelian,
            \item\label{item2} $\Sigma_0 = \pi(B)$ for some $p$-by-abelian subgroup $B \leq \Gamma$,
            \item\label{item3} $B = O_p(B) T$ for some abelian $p'$-group $T$,
            \item\label{item4} if $g \in R$ then $B^g = B^\delta$ and $T^g = T^\delta$ for some $\delta \in \pi^{-1}(O_p(\Sigma))$,
            \item\label{item5} $[R : C_R(\Sigma_0 / O_p(\Sigma))] \leq n!$.
        \end{enumerate}

        By \Cref{lem:P}(1), there is a soluble subgroup $S \leq \Gamma$
        such that $\pi(S)=\Sigma$.
        By \Cref{thm:trigonalization}, there is $S_0 \nsgp S$ of index
        $O_n(1)$ which is trigonalizable.
        By \Cref{trigonalizable-subgroups-are-PT},
        $S_0 = O_p(S_0) T_0$ for some abelian $p'$-group $T_0 \leq S_0$ isomorphic to the direct product of at most $n$ cyclic groups.
        Then $\pi(S_0) = \pi(O_p(S_0)) \pi(T_0)$ is a $p$-by-abelian normal subgroup of $\Sigma$ of index $O_n(1)$.
        Since $\pi(S_0)$ is $p$-by-abelian, its $p$-core $O_p(\pi(S_0))$ is the unique Sylow $p$-subgroup of $\pi(S_0)$, so it is normal in $\Sigma$,
        so it is contained in $O_p(\Sigma)$.
        Hence the product $\Sigma_0 = O_p(\Sigma) \pi(T_0)$ is a $p$-by-abelian subgroup of index $O_n(1)$.
        In particular, since $T_0$ is $n$-generated, $\Sigma / O_p(\Sigma)$ is $O_n(1)$-generated,
        so there are only $O_n(1)$ subgroups of $\Sigma$ containing $O_p(\Sigma)$ of index $[\Sigma : \Sigma_0]$.
        Their intersection $\Sigma_1$ is characteristic and $p$-by-abelian.
        This proves \ref{item1} with $\Sigma_1$ in the role of $\Sigma_0$.

        Write $\Sigma_1 = \Gamma_1 / N$ and $O_p(\Sigma) = \Delta / N$.
        Let $P$ be a Sylow $p$-subgroup of $\Delta$.
        Then $\pi(P) = O_p(\Sigma)$, so $\Delta = P N$, and by the Frattini argument
        \[
            \Gamma_1 = N_{\Gamma_1}(P) \Delta = N_{\Gamma_1}(P) PN = N_{\Gamma_1}(P) N.
        \]
        Hence $\pi(N_{\Gamma_1}(P)) = \Sigma_1$.
        Since $\Gamma_1 / \Delta$ is a $p'$-group, $N_{\Gamma_1}(P) / P$ is a $p'$-group, so $N_{\Gamma_1}(P) = PH$ for some $p'$-group $H$ by Schur--Zassenhaus.
        Since $H$ is a $p'$-group, Jordan's theorem (see \cite{wehrfritz}*{Theorems~9.2 and 9.3}) implies that $H$ has an abelian subgroup $H_0$ of index $O_n(1)$.
        Being an abelian $p'$-subgroup of $\GL_n(\F)$, $H_0$ is $n$-generated (e.g., by \Cref{trigonalizable-subgroups-are-PT}), and $[H : H_0]$ is $O_n(1)$,
        so $H$ is $O_n(1)$-generated and has only $O_n(1)$ subgroups of index $[H : H_0]$.
        Thus, replacing $H_0$ with the intersection of these subgroups, we may assume $H_0 \csgp H$.
        The image $\pi(PH_0)$ also has index $O_n(1)$ in $\pi(PH) =\Sigma_1$.
        Let $\Sigma_2 = \Gamma_2 / N$ be the intersection of all subgroups of $\Sigma_1$ containing $O_p(\Sigma)$ of index $[\Sigma_1 : \pi(PH_0)]$.
        Then $O_p(\Sigma) \leq \Sigma_2 \csgp \Sigma_1$, $[\Sigma_1 : \Sigma_2] \leq O_n(1)$, and $\Sigma_2 \leq \pi(PH_0)$.
        Let $T = H_0 \cap \Gamma_2$ and $B = PT = PH_0 \cap \Gamma_2$.
        Then $\pi(B) = \Sigma_2$,
        so \ref{item1}, \ref{item2}, and \ref{item3} hold with $\Sigma_2$ in the role of $\Sigma_0$.

        The normalizer $R = N_L(\Sigma)$ preserves $\Gamma$, $N$, and $\Sigma$ (by definition),
        the characteristic subgroups $O_p(\Sigma) \csgp \Sigma_2 \csgp \Sigma_1 \csgp \Sigma$,
        and also their preimages $\Delta \nsgp \Gamma_2 \nsgp \Gamma_1 \nsgp \Gamma$.
        Additionally $N_R(P)$ preserves $P$ and $N_{\Gamma_1}(P) = PH$
        as well as $PH_0 \csgp PH$
        and $PT = PH_0 \cap \Gamma_2$.
        By Sylow's theorem $\Delta$ is transitive on its Sylow $p$-subgroups, so $R = N_R(P) \Delta$.
        By Schur--Zassenhaus, $P$ is transitive on its complements in $PT$,
        so $N_R(P) = N_{N_R(P)}(T) P$.
        Thus
        \[
            R = N_R(P) \Delta = N_{N_R(P)}(T) P \Delta = N_{N_R(P)}(T) \Delta.
        \]
        This proves \ref{item4}.

        Finally, let $W = R / C_R(\Sigma_2 / O_p(\Sigma))$.
        Then $W$ acts faithfully on $\Sigma_2 / O_p(\Sigma) = \Gamma_2 / \Delta$.
        By \ref{item4}, $R = N_R(T) \Delta$.
        On the other hand since $\Sigma_2 = O_p(\Sigma) \pi(T)$ we have $C_R(T) \Delta \leq C_R(\Sigma_2 / O_p(\Sigma))$.
        Hence $W$ is a section of $N_R(T) \Delta / C_R(T) \Delta$,
        which is isomorphic to a section of $N_L(T) / C_L(T)$:
        indeed, there is a natural surjective map $N_R(T)/C_R(T) \to N_R(T)\Delta / C_R(T) \Delta$ as well as a natural injective map $N_R(T) / C_R(T) \to N_L(T) / C_L(T)$.
        Hence it suffices to prove $[N_L(T) : C_L(T)] \leq n!$.

        Since $T$ is an abelian $p'$-group, it is diagonalizable over $\bar \F$ (see, e.g., \cite{wehrfritz}*{Corollaries~1.3 and 1.6}).
        Let $V_1, \dots, V_k$ be the eigenspaces.
        Then $k \leq n$ and
        $C_L(T)$ consists of those elements of $L$ which map each $V_i$ into itself.
        Elements of $N_L(T)$ must permute $V_1, \dots, V_k$.
        Thus $N_L(T) / C_L(T)$ permutes $V_1, \dots, V_k$ faithfully,
        so $[N_L(T) : C_L(T)] \leq k! \leq n!$.
        This proves \ref{item5}.
    \end{proof}

    Although the bound $[R : C_R(\Sigma_0 / O_p(\Sigma))] \leq n!$ is all we need in this paper, the following more general result is included for independent interest.
    If $\Sigma$ is a section of $L = \GL_n(\F)$ then the \emph{Weyl group} of $\Sigma$ is $W(\Sigma) = N_L(\Sigma) / C_L(\Sigma)$.

    \begin{corollary}
        If $\Sigma$ is an abelian $p'$-section of $\GL_n(\F)$ then the Weyl group $W(\Sigma)$ has order $O_n(1)$.
    \end{corollary}
    \begin{proof}
        Let $W = W(\Sigma)$.
        By \Cref{prop:trig+}, there is a trigonalizable subgroup $\Sigma_0 \csgp \Sigma$ of index $O_n(1)$ such that $[W : C_W(\Sigma_0)] \leq n!$.
        Since $W / C_W(\Sigma / \Sigma_0)$ is isomorphic to a subgroup of $\Aut(\Sigma / \Sigma_0)$, it has order $O_n(1)$.
        Hence $W_0 = C_W(\Sigma_0) \cap C_W(\Sigma / \Sigma_0)$ has index $O_n(1)$ in $W$.
        But $W_0$ can be identified with a subgroup of $\Hom(\Sigma / \Sigma_0, \Sigma_0)$.
        Since $\Sigma_0$ is a direct product of at most $n$ cyclic groups, it follows that $W_0$ has order $O_n(1)$.
    \end{proof}

    \section{Main proof part 1: general to soluble}
    \label{sec:general-to-soluble}

        \subsection{Structure of finite linear groups}

        In this section we need a great deal of information about finite groups of Lie type. A good reference is \cite{kleidman-liebeck}*{Chapter~5}.

        Recall from the introduction that $\Lie(p)$ is the class of finite simple groups of Lie type of characteristic $p$
        and $\Lie^*(p)$ is the class of finite direct products of members of $\Lie(p)$.
        We will write $\Lie^n(p)$ for the class of direct products of at most $n$ simple groups of Lie type of characteristic $p$ and of rank at most $n$.
        A \emph{quasisimple group} (of Lie type of characteristic $p$) is a perfect central extension of some finite simple group (of Lie type of characteristic $p$).

        \begin{lemma}
            \label{lem:lies}
            Suppose $\Sigma = \Gamma / N$ is a $\Lie^*(p)$ section of $\GL_n(\F)$, where $\F$ is a finite field of characteristic $p$.
            Then $\Sigma$ is $\Lie^n(p)$.
            More precisely, $\Sigma$ has at most $n/2$ simple factors and each has Lie rank at most $n-1$.
        \end{lemma}
        \begin{proof}
            The bound on the number of factors is proved in \cite{liebeck-pyber}*{Corollary~3.3},
            while the bound on the Lie ranks follows from \cite{feit-tits} and \cite{kleidman-liebeck}*{Proposition~5.2.12}.
            Indeed, let $S$ be one of the $\Lie(p)$ factors of $\Sigma$ and let $\gamma : \Gamma \to S$ be the natural projection homomorphism. Let $H$ be a minimal subgroup of $\Gamma$ such that $\gamma(H) = S$ and let $\lambda : H \to \PGL_m(\F)$ be a nontrivial projective representation of $H$ of minimal degree. Then $m \le n$.
            By the main result of \cite{feit-tits} (see also \cite{feit-tits}*{Proposition~4.1}), $\lambda$ factorizes through $S$.
            Thus $S$ embeds in $\PGL_m(\F)$, and so \cite{kleidman-liebeck}*{Proposition~5.2.12(i)} implies that $S$ has Lie rank at most $m-1$.
        \end{proof}

        We will need the following deep theorem on the structure of finite linear groups.
        It was proved by Weisfeiler~\cite{weisfeiler} using the classification of finite simple groups,
        and later by Larsen and Pink~\cite{larsen-pink} without the classification.

        \begin{theorem}[Weisfeiler~\cite{weisfeiler}*{Theorem~1}, Larsen--Pink~\cite{larsen-pink}*{Theorem~0.2}]
            \label{thm:weisfeiler}
            Let $\F$ be a field of characteristic $p > 0$ and let $G$ be a finite subgroup of $\GL_n(\F)$.
            Then $G$ has a normal subgroup $\Gamma$ of index $O_n(1)$ containing $O_p(G)$
            such that $\Gamma / O_p(G)$ is a central extension of a member of $\Lie^*(p)$.
        \end{theorem}

        \begin{remark}
        Unfortunately this statement is not quite explicit in either \cite{weisfeiler}*{Theorem~1} or \cite{larsen-pink}*{Theorem~0.2}.

        According to \cite{weisfeiler}*{Theorem~1}, $G / O_p(G)$ has a subgroup $\Gamma = TL$ of index $O_n(1)$ such that $T$ is an abelian $p'$-group and $L$ is a central extension of a group of $\Lie^*(p)$ type. However, it is explicit in the proof on p.~5279 (though not in the statement of the theorem) that $T$ is the centre of $\Gamma$, so $\Gamma$ itself is a central extension of a $\Lie^*(p)$ group.

        Similarly, according to \cite{larsen-pink}*{Theorem~0.2}, $G$ has normal subgroups $\Gamma_3 \le \Gamma_2 \le \Gamma_1 \le G$ such that $\Gamma_1$ has index $O_n(1)$, $\Gamma_1 / \Gamma_2 \in \Lie^*(p)$, $\Gamma_2/\Gamma_3$ is an abelian $p'$-group, and $\Gamma_3$ is a $p$-group. We may assume $\Gamma_3 = O_p(G)$ by replacing $\Gamma_i$ with $\Gamma_i O_p(G)$ for $i = 1,2,3$. Now see the definition of $\Gamma_2$ on p.~1156 and the last line of the proof of Theorem~0.2 for the fact that $\Gamma_2 / \Gamma_3 = Z(\Gamma_1 / \Gamma_3)$. (See also \cite{liebeck-pyber-weisfeiler}*{Corollary~3.1}.)

        A central extension of a $\Lie^*(p)$ group is the same as a central product of an abelian group and a set of quasisimple groups of Lie type of characteristic $p$: see \cite{aschbacher}*{(31.1)}.
        \end{remark}

        \begin{corollary}
            \label{cor:weisfeiler}
            Let $\F$ be a finite field of characteristic $p > 0$ and let $P$ be a subgroup of $\GL_n(\F)$ with $\deg_\C(P)$ sufficiently large in terms of $n$. Then
            \begin{enumerate}[(a)]
                \item $P$ is soluble-by-$\Lie^*(p)$,
                \item $N = [P, \Sol(P)]$ is a $p$-subgroup such that $[P, N] = N$,
                \item $|\Sol(P) / N| \leq (2n+1)^n$.
            \end{enumerate}
        \end{corollary}
        \begin{proof}
            Apply \Cref{thm:weisfeiler} to $P$.
            Since $P$ has no proper normal subgroups of index less than $\deg_\C(P)$
            and $\deg_\C(P)$ is sufficiently large,
            it follows that $P / O_p(P)$ is a central extension of some $\Gamma \in \Lie^*(p)$.
            Hence (a) holds, and $N = [P, \Sol(P)] \leq O_p(P)$, so $N$ is a $p$-group.
            For $X$ any normal subgroup of a perfect group $P$ we have $[X, P, P] = [X, P]$ by the three subgroup lemma (see \cite{aschbacher}*{(8.9)}), so (b) holds.

            By \Cref{lem:lies}, $P$ is soluble-by-$\Lie^n(p)$.
            Suppose $P / \Sol(P) = L_1 \times \dots \times L_t$, where $t \leq n$, $L_1, \dots, L_t \in \Lie(p)$,
            and the Lie rank of $L_i$ is at most $n$ for each $i$.
            Since $P / N$ is a perfect central extension of $P / \Sol(P)$,
            \[
                |\Sol(P) / N| \leq |M(P / \Sol(P))| = \prod_{i=1}^t |M(L_i)|,
            \]
            where $M(L)$ denotes the Schur multiplier of $L$ (see \cite{aschbacher}*{Section~33 and Exercise~11.2}).
            By \cite{kleidman-liebeck}*{Theorem~5.1.4}, $|M(L_i)| \leq 2n + 1$ provided that $|L_i|$ is larger than some constant,
            which we may assume since $\deg_\C(P)$ is sufficiently large.
            Hence (c) holds.
        \end{proof}

        \begin{remark}
            In fact $N = O_p(P)$.
            This follows from the fact that, with finitely many exceptions,
            if $\Gamma$ is a finite simple group of Lie type of characteristic $p$ then $|M(\Gamma)|$ is prime to $p$
            (see \cite{kleidman-liebeck}*{Theorem~5.1.4}).
        \end{remark}

    \subsection{Reduction to soluble-by-Lie*}

    The goal of this section is to prove the upside-down version version of \Cref{thm:main-soluble}, \Cref{thm:tyukszem-intro}.
    The following statement is slightly stronger: it asserts that we may additionally assume the $\Lie^*(p)$ part of $\Gamma$ is highly quasirandom.

    \begin{theorem}
        \label{thm:tyukszem}
        Let $\F$ be a finite field of characteristic $p > 0$.
        Let $A \subset \GL_n(\F)$ be a symmetric subset such that $|A^3| \leq K|A|$, where $K \geq 2$.
        Let $d \geq 1$.
        Then there is a normal subgroup $\Gamma \nsgp \langle A \rangle$ such that
        \begin{enumerate}[(1)]
            \item $A$ is covered by $K^{O_n(d)}$ cosets of $\Gamma$,
            \item $\Gamma$ is soluble-by-$\Lie^n(p)$
            \item $\Gamma / \Sol(\Gamma)$ is covered by $A^6$, and
            \item $\deg_\C(\Gamma / \Sol(\Gamma)) \ge K^d$.
        \end{enumerate}
    \end{theorem}

    \begin{remark}
        \Cref{thm:tyukszem} is also true when $\F$ is infinite, but we do not need this generalization.
        It is even true in characteristic zero, where $\Lie^*(0)$ is interpreted as trivial;
        in this case the theorem simply states that $A$ is covered by $K^{O_n(1)}$ cosets of a soluble normal subgroup (this was first proved in \cite{BGT-linear}).
    \end{remark}

    This result is a relatively direct consequence of \Cref{thm:weisfeiler} and \Cref{thm:PS-main} (the Product Theorem).
    Indeed, already by \Cref{thm:weisfeiler}, $G = \langle A \rangle$ is covered by $O_n(1)$ cosets of
    a soluble-by-$\Lie^*(p)$ normal subgroup $\Gamma \nsgp G$.
    To get such a subgroup $\Gamma$ for which $A^6$ covers $\Gamma / \Sol(\Gamma)$ we need one auxiliary result on $\Lie^*(p)$ groups.


    \begin{lemma}
      \label{lem:direct-product}
      Let $\Gamma \in \Lie^n(p)$,
      and let $A$ be a symmetric generating set of $\Gamma$ that projects onto all simple quotients of $\Gamma$.
      Then $A^{O_n(1)} = \Gamma$.
    \end{lemma}
    \begin{proof}
        By induction it suffices to prove that if $A$ generates a group of the from $H \times L$ with $L \in \Lie(p)$
        and $A$ projects onto both $H$ and $L$
        then $A^{O(\ell)} = H \times L$, where $\ell$ is the Lie rank of $L$.
        If $A$ is a subgroup we are done, since $A$ is a generating set by hypothesis.
        Otherwise let $x \in A^2 \setminus A$ and let $y \in A$ such that $x$ and $y$ have the same projection in $H$.
        Then $a = xy^{-1}$ is a nontrivial element of $A^3 \cap L$.
        By \cite{lawther--liebeck}*{Theorem~1} (and recalling that untwisted Lie rank is at most twice the Lie rank), every element of $L$ is the product of $m$ conjugates of $a$ for some $m \leq O(\ell)$.
        Since $A$ projects onto $L$, it follows that $A^{5m}$ contains $L$.
        Since $A$ projects onto $H$, $A^{5m+1} = H \times L$.
    \end{proof}

    Now we can prove \Cref{thm:tyukszem}.

    \begin{proof}[Proof of \Cref{thm:tyukszem}]
        By \Cref{thm:weisfeiler} and \Cref{lem:lies}, $G = \langle A \rangle$ has a soluble-by-$\Lie^n(p)$ normal subgroup $\Gamma$ of index $m \le O_n(1)$.
        We claim that $G$ has a normal subgroup $\Delta$ such that $\Sol(\Gamma) \leq \Delta \leq \Gamma$
        and $\Delta / \Sol(\Gamma)$ is covered by $A^6$
        and $A$ is covered by $K^{O_n(1)}$ cosets of $\Delta$.

        By \Cref{lem:schreier}, $\Gamma$ is generated by $B = A^{3m} \cap \Gamma$.
        By \Cref{lem:growth-in-sections}, $B$ has tripling $K^{O_n(1)}$.

        Write $\Ell$ for the set of simple factors of $\Gamma / \Sol(\Gamma)$.
        For $L \in \Ell$ denote by $\pi_L : \Gamma \to L$ the natural projection.
        Since $\Gamma \nsgp G$, $G$ permutes $\Ell$.
        We claim that $\Ell = \Ell_1 \cup \Ell_2$ for $G$-invariant sets $\Ell_1, \Ell_2$ such that
        \begin{align*}
            &\pi_L(B)^3 = L~\text{and}~\deg_C(L) \ge K^d && (L \in \Ell_1) \\
            &|\pi_L(B)| \leq K^{O_n(d)} && (L \in \Ell_2).
        \end{align*}
        Let $\Ell_1 \subset \Ell$ be the largest $G$-invariant set such $\pi_L(B)^3 = L$ and $\deg_C(L) \ge K^d$ for every $L \in \Ell_1$
        and let $\Ell_2 = \Ell \setminus \Ell_1$.
        It suffices to show $|\pi_L(B)| \le K^{O_n(d)}$ for all $L \in \Ell_2$.
        Suppose $L_1 \in \Ell$ satisfies $\pi_{L_1}(B)^3 \ne L_1$.
        By \Cref{lem:growth-in-sections}, $\pi_{L_1}(B) \subset L_1$ has tripling $K^{O_n(1)}$, so \Cref{thm:PS-main} implies that $|\pi_{L_1}(B)| \le K^{O_n(1)}$.
        Similarly, if $\deg_\C(L_1) < K^d$ then $|L_1| \le \deg_\C(L_1)^{O(n)} \le K^{O_n(d)}$ by \Cref{landazuri--seitz}, so certainly $|\pi_{L_1}(B)| \le K^{O_n(d)}$.
        Now for any $L \in \Ell_1$ and $a \in A$ we have, by \Cref{lem:growth-in-sections}(3),
        \begin{align*}
            |\pi_{L^a}(B)| = |\pi_{L}(B^{a^{-1}})| \le |\pi_L(A^{3m+2} \cap \Gamma)|
            &\le K^{O_n(1)} |\pi_L(A^2 \cap \Gamma)| \\
            &\le K^{O_n(1)} |\pi_L(B)|.
        \end{align*}
        Since $|\Ell| \le n$ and $G$ is generated by $A$, it follows that $|\pi_L(B)| \le K^{O_n(d)}$ for every $L$ in the $G$-orbit of $L_1$.
        Since $\Ell_2$ is the union of such $G$-orbits, the claim holds.

        Let $\Gamma_1$ be the subgroup of $\Gamma$ corresponding to the product of the factors in $\Ell_1$.
        Since the projection of $B$ to any factor of $\Gamma / \Gamma_1$ (i.e., any $L \in \Ell_2$) has size $K^{O_n(d)}$,
        $B$ is covered by $K^{O_n(d)}$ cosets of $\Gamma_1$.
        By \Cref{lem:covering-observation}, $A$ is also covered by $K^{O_n(d)}$ cosets of $\Gamma_1$.
        Since $B^3$ projects onto each factor of $\Gamma_1 / \Sol(\Gamma)$,
        $B^k$ projects onto $\Gamma_1 / \Sol(\Gamma)$ by \Cref{lem:direct-product} for some $k \le O_n(1)$.
        In particular $A^{3mk}$ covers $\Gamma_1 / \Sol(\Gamma)$.
        Applying \Cref{lem:growth-in-sections}(3) with $\Sigma = \Gamma_1 / \Sol(\Gamma)$, it follows that
        \[
            |\tr(A^2, \Sigma)| \ge |\tr(A^{3mk}, \Sigma)| / K^{3mk+3}.
        \]
        Since $\deg_\C(\Sigma) \ge K^d$, \Cref{prop:quasirandomness} implies that $\tr(A^2, \Sigma)^3 = \Sigma$,
        provided that $d \ge 3(3mk+3)$, which we may assume.
        Thus $A^6$ covers $\Gamma_1 / \Sol(\Gamma)$.
        This completes the proof of the theorem with $\Gamma_1$ in the role of $\Gamma$.
    \end{proof}

    \subsection{Covering the perfect core}

    In this section we start with the soluble-by-$\Lie^*(p)$ group $\Gamma$ furnished by \Cref{thm:tyukszem-intro}.
    By \Cref{lem:P}, $\Gamma$ has derived length $O(\log n)$.
    Let $P = \Gamma^{(\omega)} = \Gamma^{(O(\log n))}$ be the perfect core of $\Gamma$.
    Our goal is to prove that $A^{O_n(1)}$ covers $P$.
    It is easy to show that $A^{O_n(1)}$ covers $P / \Sol(P)$, by iterating the following lemma.
    The main work of this section consists of showing that $A^{O_n(1)}$ also covers $\Sol(P)$.

    \begin{lemma}
        \label{lem:Gamma'}
        Let $\Gamma$ be soluble-by-$\Lie^*(p)$.
        Let $A \subset \Gamma$ be a symmetric set covering $\Gamma / \Sol(\Gamma)$.
        Then $A^{O(1)}$ covers $\Gamma' / \Sol(\Gamma')$.
    \end{lemma}
    \begin{proof}
        Let $\bar \Gamma = \Gamma / \Sol(\Gamma)$.
        Since $A$ covers $\Gamma / \Sol(\Gamma)$,
        the projection of $A^4 \cap \Gamma'$ to $\bar \Gamma$ contains the set of commutators $C$.
        By the weak Ore conjecture, $\bar \Gamma = C^{O(1)}$ (see \cite{wilson-pseudofinite}*{Proposition~2.4} or \cite{shalev} or \cite{nikolov-pyber}*{Theorem~3}).
        Hence $(A^4 \cap \Gamma')^{O(1)}$ covers $\Gamma / \Sol(\Gamma)$.
        Since $\Gamma' \cap \Sol(\Gamma) = \Sol(\Gamma')$, this implies that $(A^4 \cap \Gamma')^{O(1)}$ covers $\Gamma' / \Sol(\Gamma')$.
    \end{proof}

    We need the following elementary result of Rhemtulla: see \cite{rhemtulla1969commutators}*{Lemma~2}.

    \begin{lemma}[Rhemtulla]
        \label{lem:rhemtulla}
        If $\Gamma = V \rtimes G$ where $V$ is abelian and $G = \gen{x_1, \dots, x_d}$ then
        \[
            [V, G] = \{[v_1, x_1] \cdots [v_d, x_d] : v_1, \dots, v_d \in V\}.
        \]
    \end{lemma}

    \begin{lemma}[affine conjugating trick, \Cref{lem:affine-conjugating-trick-intro} restated]
        \label{lem:affine-conjugating-trick}
        Let $\Gamma = V \rtimes G$ be the semidirect product of an abelian group $V$
        and a $d$-generated finite group $G$ with $\deg_\C(G) \ge K^{21}$.
        Let $A \subset V$ be a symmetric $G$-invariant set generating $V$.
        If $|A^3| \leq K|A|$ then $A^{7d} \supset [V, G]$.
    \end{lemma}
    \begin{proof}
        Let $B = AG$.
        Since $B^n = A^nG$ for all $n\neq 0$, $|B^3| \leq K|B|$ and $B$ is a symmetric generating set for $\Gamma$.

        We claim that $B^6$ contains all conjugates of $G$.
        It suffices to prove if $G_0$ is a conjugate of $G$ contained in $B^6$ and $b \in B$ then $G_1 = G_0^b \subset B^6$.
        Certainly $G_1 \subset B^8$, so by \Cref{lem:growth-in-sections}(1),
        \[
            |G_1| = |B^8 \cap G_1|
            \leq K^7 |B^2 \cap G_1|.
        \]
        Hence $(B^2 \cap G_1)^3 = G_1$ by \Cref{prop:quasirandomness},
        so $B^6 \supset G_1$, as required.

        Hence if $g \in G$ and $v \in V$ the element $[v, g] = (g^{-1})^v g$ is contained in $B^6B \cap V = A^7$.
        Let $x_1, \dots, x_d$ be generators of $G$.
        Then
        \[
            A^{7d} \supset \{[v_1, x_1] \cdots [v_d, x_d] : v_1, \dots, v_d \in V\}.
        \]
        Finally,
        $\{[v_1, x_1] \cdots [v_d, x_d] : v_1, \dots, v_d \in V\} = [V, G]$
        by \Cref{lem:rhemtulla}.
    \end{proof}

    \begin{proposition}
        \label{prop:GN-prop}
        For each $d \geq 0$ there is a constant $m = m(d)$ such that the following holds.
        Let $N \leq G$ be finite normal subgroups of a group $\Gamma$ such that
        \begin{enumerate}[(i)]
            \item $[G, N] = N$,
            \item $N$ is nilpotent,
            \item $G/N$ is $d$-generated, and
            \item $\deg_\C(G) \ge K^m$,
        \end{enumerate}
        where $K \ge 2$.
        Let $A$ be a finite symmetric $K$-tripling set generating $\Gamma$ and covering $G / N$.
        \begin{enumerate}[(1)]
            \item If $G = \Gamma$ then $A^3 = G$.
            \item In general $(A^2 \cap G)^3 = G$.
        \end{enumerate}
    \end{proposition}
    \begin{proof}
        Call $A$ \emph{hereditarily $K$-tripling} if $|(A^\pi)^3| \leq K|A^\pi|$ for every quotient $\pi : \Gamma \to \Gamma/\ker \pi$ of $\Gamma$.
        It follows from \Cref{lem:growth-in-sections}(2) that $A$ is hereditarily $K^3$-tripling.
        Hence it suffices to prove the proposition for hereditarily $K$-tripling sets.
        This observation allows us to use induction.

        We begin with (1).
        We argue by induction on $|N|$.
        If $N = 1$ then the claim is trivial because $A$ covers $G/N$ by hypothesis, so assume $N \neq 1$.
        If $V$ is a nontrivial normal subgroup of $G$ contained in $N$ then the hypotheses hold for $G / V$, so by induction $A^3 V = G$.

        Suppose we can find a nontrivial commutator $[a, b] \in Z(G) \cap N$.
        Let
        \[
            V = \langle [a, b] \rangle = \{[a^k, b] : k \in \Z\}.
        \]
        Then $V$ is normal in $G$ and contained in $N$ so by induction $G = A^3 V$.
        In particular $G = A^3 Z(G)$, so all commutators are contained in $A^{12}$.
        In particular $V \subset A^{12}$, so $G = A^{15}$.
        Hence $A^3 = G$ by \Cref{lem:gowers-trick}, assuming $m \geq 45$.

        Hence assume no nontrivial commutator is in $Z(G) \cap N$.
        Since $[G, N] = N$, $N$ is not contained in $Z(G)$.
        Let $V \leq N$ be a minimal normal subgroup of $G$ not contained in $Z(G)$.
        If $[V, G] < V$ then $[V, G] \leq Z(G) \cap N$ by minimality of $V$, but this contradicts the assumption that no nontrivial commutator is in $Z(G) \cap N$.
        Since $N$ is nilpotent, $[V, N] < V$, for otherwise $\gamma_i(N) \ge V$ for all $i$.
        Hence $[V, N] \le Z(G) \cap N$ by minimality of $V$ again, a contradiction unless $[V, N] = 1$.
        Hence $[V, G] = V$ and $V \leq Z(N)$.

        In particular $V$ is an abelian group (since $V \le Z(N)$) and the conjugation action of $G$ on $V$ factors through $G/N$ (since $[V, N] = 1$), so we may identify $V$ with a $\Z(G/N)$-module satisfying $[V, G/N] = V$.

        Next we note that $A^7$ contains an element $x \in V \setminus Z(G)$.
        If $V \cap Z(G) \neq 1$ then by induction $A^3 (V \cap Z(G)) = G$, so $A^3$ must contain an element of $V \setminus Z(G)$.
        If $V \cap Z(G) = 1$ then, since $A^3 V = G$ and $A$ generates
        $G$, $A^4$ contains two elements of some coset of $V$ and
        $A^3$ also intersects this coset, so $A^7$ contains a nontrivial element of $V$, so we are done.

        By minimality of $V$, $x$ generates $V$ as a $\Z(G/N)$-module.
        Hence $A^7 \cap V$ also generates $V$.
        Since $A$ covers $G/N$, taking the union of all $A$-conjugates of $A^7 \cap V$ produces a symmetric $G/N$-invariant generating set $B$ of $V$ contained in $A^9$, and
        \[
            \frac{|B^3|}{|B|}
            \leq \frac{|A^{27} \cap V|}{|A^7 \cap V|}
            \leq \frac{|A^{27} \cap V|}{|A^2 \cap V|}
            \leq K^{26}
        \]
        by \Cref{lem:growth-in-sections}.
        Hence by \Cref{lem:affine-conjugating-trick}, $B^{7d} = [V, G/N] = V$.
        Hence $A^{63d} \supset V$ and $G = A^3 V = A^{63d+3}$.
        Again \Cref{lem:gowers-trick} implies $G = A^3$ provided $m > 3(63d + 3)$.

        Next we prove (2).
        By \Cref{lem:growth-in-sections}(1), $A^2 \cap G$ is $K^5$-tripling.
        Let $H = \langle A^2 \cap G \rangle$.
        By (1) it suffices to prove $H = G$.

        We again argue by induction on $|N|$.
        If $N = 1$ then $H = G$ because $A$ covers $G/N = G$, so  assume $N \neq 1$.
        If $V$ is a nontrivial $A$-invariant subgroup of $N$
        then the hypotheses hold for $\Gamma / V$, so by induction $HV = G$.

        Suppose $H < G$.
        Let $V = Z(G) \cap N$.
        If $V \neq 1$ then $HV = G$, so $H \nsgp G$ and $G/H$ is abelian, in contradiction to $\deg_\C(G) > 1$.
        Hence $Z(G) \cap N = 1$.

        Since $N$ is nilpotent, $Z(N) \neq 1$.
        Let $V$ be a minimal $A$-invariant subgroup of $Z(N)$ and $U = V \cap H$.
        Since $V \nsgp \Gamma$, $U \nsgp H$.
        Since $G = HV$, $U \nsgp G$.

        By the modular law, $N = N \cap HV = (N \cap H) V$.
        Hence $N = [G, N]$ implies
        \[
          (N \cap H) V = [HV, (N \cap H) V] \le [H, N \cap H] V.
        \]
        Moreover, $[H,N\cap H] \le [HV, (N \cap H) V]$ and $V \le (N\cap H)V$,
        hence
        \[
          (N \cap H) V = [H, N \cap H] V.
        \]
        Intersecting with $H$ and applying the modular law again,
        \begin{equation}
            \label{eq:NcapH1}
            N \cap H = [H, N \cap H] U.
        \end{equation}

        We may identify $V$ with a simple $\Z(\Gamma/N)$-module.
        Since $G \nsgp \Gamma$, $V$ is a semisimple $\Z(G/N)$-module by Clifford's theorem.
        Since $U \nsgp G$, $U$ is a $\Z(G/N)$-submodule and hence also semisimple.
        If $S$ is any simple submodule of $U$ then $[G/N, S]$ is a submodule of $S$,
        and it cannot be trivial since $Z(G) \cap N = 1$,
        so $S = [G/N, S]$.
        Hence also $U = [G/N, U]$.
        Since $H$ covers $G/N$, $U = [H, U] \leq [H, N \cap H]$.

        Hence, from \eqref{eq:NcapH1},
        \[
            N \cap H = [H, N \cap H] \leq [H, H].
        \]
        Since $H / (N \cap H) \cong G / N$ and $G$ is perfect (since $\deg_\C(G) > 1$), $H$ is perfect.
        By \Cref{lem:degrees} it follows that $\deg_\C(H) \ge K^{cm}$ (recall that $K \ge 2$).
        Hence by (1) applied to $A^2 \cap G$ and $H$ it follows that $H = (A^2 \cap G)^3 \subset A^6$.

        Since $\deg_\C(H) \ge K^{cm}$, every proper subgroup of $H$ has index at least $K^{cm}$.
        If $a \in A$, $H^a \subset A^8$.
        On the other hand $A^2 \cap H^a = A^2\cap G\cap H^a \subset H \cap H^a$.
        Hence
        \[
            [H : H \cap H^a] = \frac{|H^a|}{|H \cap H^a|}
            \leq \frac{|A^8 \cap H^a|}{|A^2 \cap H^a|}
            \leq K^7
        \]
        by \Cref{lem:growth-in-sections}.
        It follows that $H = H^a$.
        Hence $H \nsgp G$.
        Since $G = H Z(N)$, $G/H$ is abelian and perfect, hence trivial.
    \end{proof}

    We are now ready to prove \Cref{thm:main-soluble}.

    \begin{proof}[Proof of \Cref{thm:main-soluble}]
        By \Cref{prop:main-soluble-finitization} we may assume $\F$ is a finite field.
        By \Cref{thm:tyukszem}, there is $\Gamma \nsgp \langle A \rangle$ such that
        $A$ is covered by $K^{O_n(1)}$ cosets of $\Gamma$,
        $\Gamma / \Sol(\Gamma)$ is covered by $A^6$, and
        $\Gamma / \Sol(\Gamma) \in \Lie^n(p)$,
        and $\deg_\C(\Gamma / \Sol(\Gamma)) > K^d$,
        where $d = d(n)$ is sufficiently large for the rest of the argument.

        By \Cref{lem:P}, $\Gamma$ has derived length $O(\log n)$.
        Let $P = \Gamma^{(O(\log n))}$ be the perfect core.
        By iterating \Cref{lem:Gamma'},  $A^{O_n(1)}$ covers $P / \Sol(P)$.
        Note that $P / \Sol(P) \cong \Gamma / \Sol(\Gamma)$ since $\Gamma / \Sol(\Gamma)$ is perfect.

        By \Cref{lem:degrees}, $\deg_\C(P) \geq c K^{d/2}$.
        By \Cref{cor:weisfeiler}, there is a normal $p$-subgroup $N \nsgp P$ contained in $\Sol(P)$ such that $[P, N] = N$ and $|\Sol(P)/N| \leq (2n+1)^n$.
        Since $A^{O_n(1)}$ covers $P / \Sol(P)$, it projects to a subset of $P / N$ of size at least $|P / \Sol(P)| \geq (2n+1)^{-n} |P / N|$.
        Assuming $\deg_\C(P)$ is larger than $(2n+1)^{3n}$, this implies that $A^{O_n(1)}$ covers $P / N$ by \Cref{prop:quasirandomness}.
        By the fact that every finite simple group is $2$-generated, $P / \Sol(P)$ is $2n$-generated,
        so $P / N$ is $O_n(1)$-generated.
        Applying \Cref{prop:GN-prop} to the section $P/N$, it follows that $A^{O_n(1)}$ covers $P$.
        Finally, \Cref{lem:gowers-trick} implies that $A^3$ contains a coset of $P$, and the proof is complete.
    \end{proof}

    \section{Main proof part 2: soluble to nilpotent}
    \label{sec:soluble-to-nilpotent}

    \subsection{Pivoting}
    \label{sec:pivoting}

    The workhorse of the rest of the proof is a pivoting argument due to Gill and Helfgott related to sum-product theory.
    Let $T$ be an abelian group acting on a nontrivial group $U$ by automorphisms.
    We use multiplicative notation in $T$ and $U$ and exponential notation $(t, u) \mapsto u^t$ for the action of $T$ on $U$.
    For $W \subset U$ and $X \subset T$ we write $W^X$ for $\{w^x : w \in W, x \in X\}$.
    Let $F = F(T, U) \subset T$ be the set of $t \in T$ having a fixed point in $U \setminus \{1\}$.

    \begin{proposition}[\cite{gill-helfgott}*{Proposition~2.11}]
        \label{prop:pivoting}
        Let $X \subset T$ and $W \subset U$. Then either
        \begin{equation}
            \label{eq:GH-growth}
            |(W^{X^{\pm 2}})^{\pm 6}| \geq \frac12 \frac{|W||X|}{|X^{-1} X \cap F|}
        \end{equation}
        or
        \begin{equation}
            \label{eq:GH-stability}
            (W^X)^{\pm 8} = \langle W^{\langle X\rangle}\rangle.
        \end{equation}
    \end{proposition}

    \begin{remark}
        The factor of 1/2 in \eqref{eq:GH-growth} does not appear in \cite{gill-helfgott}*{Proposition~2.11},
        but \cite{gill-helfgott}*{(2.5)} appears to be unjustified
        (the argument given shows only $|Y|^2 |W|^2 \geq | \langle \langle X \rangle (\langle W\rangle)\rangle|$).
        Omitting this inequality, the rest of the proof is only impacted by a factor of 2, arguing as in \cite{helfgott-SL3}*{Proposition~3.1}.

        The statement \cite{gill-helfgott}*{Proposition~2.11} also assumes that $T$ acts faithfully on $U$, but the proof does not use this hypothesis.
    \end{remark}

    The idea of the rest of this section is to apply \Cref{prop:pivoting} in appropriate trigonalizable sections of $\GL_n(\F)$, or equivalently quotients of trigonalizable subgroups. The material is somewhat technical and the reader is encouraged to keep in mind the case of trigonalizable subgroups as a representative case. However, to prove our main theorem, the more general case of trigonalizable sections seems to be necessary.

    Let $\Sigma = \Gamma / N$ be a trigonalizable section of $\GL_n(\F)$, where $\F$ is a finite field of characteristic $p$.
    By definition this means that $\Gamma = BN$ for some trigonalizable subgroup $B$.
    \Cref{trigonalizable-subgroups-are-PT} implies that $B$ and hence $\Sigma$ is $p$-by-abelian.
    By Schur--Zassenhaus we therefore have a semidirect decomposition $\Sigma = UT$ where $U = O_p(\Sigma)$ and $T$ is some abelian $p'$-group.
    Note that $\Sigma$ acts naturally on $U$ by conjugation.
    Let $V$ be a $\Sigma$-composition factor of $U$ and let $K = C_\Sigma(V)$.
    Then $U$ acts trivially on $V$ (as in the proof of \Cref{trigonalizable-subgroups-are-PT}), so $U \le K$ and $\Sigma / K \cong T / T \cap K$ is an abelian $p'$-group.
    We call $K$ a \emph{root kernel} and the corresponding homomorphism $\chi : \Sigma \to \Sigma / K$ a \emph{root}.
    The set of nontrivial roots of $\Sigma$ is denoted $\roots^*(\Sigma)$.

    \begin{lemma}
        \label{lem:n^2-roots}
        \leavevmode
        \begin{enumerate}[(1)]
            \item If $K$ is a root kernel then $\Sigma / K$ is isomorphic to a subgroup of $\F^\times$. Thus we may assume roots take values in $\F^\times$.
            \item If $V$ is a $\Sigma$-composition factor of $U$ and $\chi : \Sigma \to \F^\times$ is the corresponding root then $V \cong \F_p(\chi(T))$, the subfield of $\F$ generated by the image of $\chi$ with the action $v^g = \chi(g) v$.
            \item $|\Phi^*(\Sigma)| < n^2$.
        \end{enumerate}
    \end{lemma}
    \begin{proof}
        By hypothesis $\Gamma = BN$ for some trigonalizable subgroup $B \leq \GL_n(\F)$.
        By replacing $\F$ with an extension and $\Sigma$ with a conjugate we may assume $B \leq B_n(\F)$.
        Then $B$ acts on $U_n(\F)$, $O_p(B) \le U_n(\F)$, and $U \cong O_p(B) / O_p(B) \cap N$.
        By Jordan--H\"older, the $\Sigma$-composition factors (equivalently, $B$-composition factors) of $U$ appear among the $B$-composition factors of $U_n(\F)$.

        Consider the following $B$-invariant series for $U_n(\F)$.
        Let $U_d$ be the subgroup of all $g \in U_n(\F)$ such that $g_{ij} = 0$ for $0 < j - i < d$.
        Then $U_d / U_{d+1}$ is a direct sum of copies $V_{ij}$ of $\F$, for $j - i = d$, where $B$ acts on $V_{ij}$ according to $v^g = \chi_{ij}(g) v$, where $\chi_{ij}(g) = g_{ii} / g_{jj}$.
        Note that $V_{ij}$ is a direct sum of isomorphic copies of the irreducible $B$-module $\F_p(\chi_{ij}(B))$,
        and $C_B(v) = \ker \chi_{ij}$ for every nonzero $v \in V_{ij}$.

        Thus if $V$ is a $\Sigma$-composition factor of $U$ then $V \cong \F_p(\chi_{ij}(B))$ for some $i, j$.
        If $K$ is the corresponding root kernel then $K = \ker \chi_{ij} N / N$ (and $B \cap N \le \ker \chi_{ij}$, since $B \cap N$ must act trivially).
        Since there are fewer than $n^2$ possibility for $(i, j)$, this proves the three claims.
    \end{proof}

    \begin{lemma}
        Let $V$ be a $T$-invariant section of $U = O_p(\Sigma)$ such that $C_V(T) = 1$.
        Then
        \[
            F(T, V) \subset \bigcup_{\chi \in \roots^*(\Sigma)}  \ker \chi.
        \]
    \end{lemma}
    \begin{proof}
        Suppose $t \in F(T, V)$, so there is some nontrivial $v \in V$ such that $v^t = v$.
        By replacing $V$ with $\langle v^T\rangle$ we may assume $V = \langle v^T\rangle$.
        Since $V$ is nontrivial, $V \ne \frat(V)$, so $v^T \not\subset \frat(V)$, which implies that $v \notin \frat(V)$.
        By replacing $V$ with $V / \frat(V)$ we may therefore assume $\frat(V)$ is trivial
        (by \Cref{lem:U-decomp}, $V = [V, T]$, so the quotient has the same property, so the condition $C_V(T) = 1$ is preserved).
        Hence $V$ is elementary abelian and we may identify it with an $\F_pT$-module.
        By Maschke's theorem, $V$ is completely reducible.
        By projecting to one of the irreducible components we may assume $V$ is irreducible.
        Note then $[V, U] = 1$, so $V$ is a $\Sigma$-composition factor of $U$.
        Now $C_V(t)$ is a nontrivial submodule of $V$, so $C_V(t) = V$, so $t$ is contained in the root kernel $K = C_\Sigma(V)$, and $K \neq \Sigma$ because $C_V(T) = 1$.
        Hence $t$ is contained in a nontrivial root kernel.
    \end{proof}

    \pagebreak[2]
    \begin{proposition}
        \label{section-pivoting}
        Let $V$ be a $\Sigma$-invariant section of $U = O_p(\Sigma)$ such that $[V, U] = C_V(T) = 1$.
        Let $A \subset \Sigma$ be a symmetric $K$-tripling subset such that
        $\langle A^\pi \rangle = T$, where $\pi : \Sigma \to T$ is the natural projection.
        Then either
        \begin{enumerate}[(a)]
            \item there is some $\chi \in \roots^*(\Sigma)$ such that $1 < |\chi(A)| \leq
            2n^2 K^{O(1)}$, or
            \item $A^{O(1)}$ covers $\langle (A^2 \cap V)^T \rangle$.
        \end{enumerate}
    \end{proposition}
    \begin{proof}
        Assume that $|\chi(A)| \geq R$ for every $\chi \in \roots^*(\Sigma)$. Then by \Cref{lem:orbit-stab} and \Cref{lem:growth-in-sections}(2),
        \[
            |(A^\pi)^2 \cap \ker \chi| \leq R^{-1} |(A^\pi)^3| \leq R^{-1} K^3 |A^\pi|.
        \]
        Hence by the previous two lemmas
        \[
            |(A^\pi)^2 \cap F(T, V)| \leq n^2 K^3 R^{-1} |A^\pi|.
        \]
        Since $[V, U] = 1$, the action of $\Sigma$ on $V$ factors through $\pi : \Sigma \to T$.
        Apply \Cref{prop:pivoting} with $W = A^2 \cap V$ and $X = A^\pi$.
        If \eqref{eq:GH-growth} holds then
        \[
            |((A^2 \cap V)^{A^2})^6|
            \geq \frac12 \frac{| A^2 \cap V| |A^\pi|}{|(A^\pi)^2 \cap F(T, V)|}
            \geq \frac12 n^{-2} K^{-3} R |A^2 \cap V|,
        \]
        while
        \[
            |((A^2 \cap V)^{A^2})^6|
            \leq |A^{36} \cap V|
            \leq K^{35} |A^2 \cap V|
        \]
        by \Cref{lem:growth-in-sections}(1).
        This implies (a).
        On the other hand if \eqref{eq:GH-stability} holds then
        \[
            \langle (A^2 \cap V)^T\rangle = ((A^2 \cap V)^A)^8 \subset A^{32},
        \]
        which implies (b).
    \end{proof}

    \subsection{Growth of bilinear images}

    In this section we consider abelian groups only, so we use additive notation.

    We briefly recall the definition of the Fourier transform on a finite abelian group $G$.
    The dual group $\hat G$ is the group of all homomorphisms $\chi : G \to S^1$.
    We endow $G$ with the counting measure and $\hat G$ with the uniform measure.
    The Fourier transform of a function $f : G \to \C$ is then defined by
    \[
        \hat f(\chi) = \sum_{x \in G} f(x) \chi(-x) \qquad (\chi \in \hat G),
    \]
    and the Fourier inversion formula is
    \[
        f(x) = \frac1{|G|}\sum_{\chi \in \hat G} \hat f(\chi) \chi(x) \qquad (x \in G).
    \]
    Parseval's identity is
    \[
        \sum_G |f|^2 = \frac1{|G|} \sum_{\hat G} |\hat f|^2.
    \]

    The convolution of two functions $f_1, f_2 : G \to \C$ is defined by
    \[
        f_1 * f_2(x) = \sum_{y \in G} f_1(y) f_2(x-y).
    \]
    With this definition we have the rule
    \[
        \hat {f_1 * f_2} = \hat {f_1} \hat{f_2}.
    \]

    In the following lemma, by a \emph{probability measure} on $G$ we simply mean a function $\mu\colon G \to [0, 1]$ such that $\sum_{x \in G} \mu(x) = 1$.

    \begin{lemma}
        Let $G$ be a finite abelian group,
        and let $\mu$ be a probability measure on $G$ such that, for all $\chi \in \hat G$, either $|\hat\mu(\chi)| \leq 1/R$ or $\hat\mu(\chi) = 1$.
        Let $S$ be the support of $\mu$, and let $A \subset G$.
        Then
        \[
            |A + S| > \frac12 \min(R^2 |A|, |\langle S \rangle|).
        \]
        Moreover, if $S \subset A \subset \langle S\rangle$ and $|2A| \leq \frac12 R^2 |A|$, then $4A = \langle S \rangle$.
    \end{lemma}
    \begin{proof}
        The convolution $1_A * \mu$ is supported on $A+S$, so, by Cauchy--Schwarz and Parseval's identity,
        \[
            |A|^2
            = (\sum_G 1_A * \mu)^2
            \leq |A + S| \sum_G |1_A * \mu|^2
            = |A + S| \frac1{|G|} \sum_{\hat G} |\hat{1_A}|^2 |\hat\mu|^2.
        \]
        Let $H = \langle S \rangle$, and note $\hat\mu(\chi) = 1$ if and only if $\chi \in H^\perp = \{\psi \in \hat G : \psi(H) = 1\}$.
        Hence, using Parseval's identity again,
        \begin{align*}
            \frac1{|G|} \sum_{\hat G} |\hat{1_A}|^2 |\hat\mu|^2
            &\leq \frac{|H^\perp|}{|G|} |A|^2 + R^{-2} \frac1{|G|} \sum_{\hat G \setminus H^\perp} |\hat {1_A}|^2 \\
            &< \frac{|A|^2}{|H|} + R^{-2} |A|.
        \end{align*}
        Rearranging,
        \[
            |A + S| > \frac{|A|}{|A| / |H| + R^{-2}}
            \geq \frac12 \min(|H|, R^2|A|).
        \]

        For the last statement, if $|A+S| \leq \frac12 R^2 |A|$ then $|A+S| > \frac12 |H|$, so $4A \supset 2(A+S) = H$ by a standard exercise.
    \end{proof}

    \begin{proposition}
    \label{prop:bilinear}
    Let $U, V, Z$ be finite abelian groups,
    let $T$ be a group acting on $U$ and $V$,
    and let $\beta : U \times V \to Z$ be a bilinear map which is $T$-invariant in the sense that
    \[
        \beta(u^t, v^t) = \beta(u, v) \qquad (u \in U, v \in V, t \in T).
    \]
    Suppose that the $T$-simple composition factors of $U$ have size at least $R$.
    Write $\beta(U, V)$ for the image of $U \times V$ and let $W = \langle \beta(U, V)\rangle$.
    Then for any $A \subset Z$ we have
    \[
        |A + \beta(U, V)| > \frac12 \min(R^2 |A|, |W|).
    \]
    Moreover, if $\beta(U, V) \subset A \subset W$
    and $|2A| \leq \frac12 R^2 |A|$ then $4A=W$.
    \end{proposition}
    \begin{proof}
        Let $\mu$ be the pushforward of the uniform measure on $U \times V$,
        so that
        \[
            \mu(z) = \frac{\#\{(u, v) \in U \times V : \beta(u, v) = z\}}{|U| |V|}.
        \]
        Then for $\chi \in \hat Z$,
        \begin{align*}
            \hat \mu(\chi)
            = \sum_{z \in Z} \mu(z) \chi(z)
            = \frac1{|U||V|} \sum_{u \in U, v \in V} \chi(\beta(u, v))
            = [U : U_\chi]^{-1},
        \end{align*}
        where $U_\chi$ is the subgroup of all $u \in U$ such that $\chi(\beta(u, V)) = 1$.
        Since $U_\chi$ is $T$-invariant, either $U_\chi = U$ or $[U:U_\chi] \geq R$, so the lemma applies.
    \end{proof}

    \subsection{The no-small-roots case}
    \label{sec:no-small-roots}

    Suppose $\Sigma = \Gamma / N$ is a trigonalizable section of $\GL_n(\F)$.
    Recall from \Cref{sec:pivoting} that $\roots^*(\Sigma)$ is the set of nontrivial roots of $\Sigma$.
    Call $\chi \in \roots^*(\Sigma)$ an \emph{$R$-small root} for $A \subset \Sigma$ if $1 < |\chi(A)| \leq R$.
    In this section we establish the trigonal no-small-roots case of \Cref{thm:main}.

    \begin{proposition}
        \label{prop:nsr-case}
        There is a constant $C_n$ such that the following holds.
        Let $\F$ be a finite field.
        Let $\Sigma = \Gamma / N$ be a trigonalizable section of $\GL_n(\F)$.
        Let $A \subset \Sigma$ be a nonempty symmetric set containing $1$ such that
        \begin{enumerate}[(i)]
            \item $|A^3| \leq K|A|$,
            \item $A$ has no $K^{C_n}$-small roots,
            \item $A$ generates $\Sigma$.
        \end{enumerate}
        Then $\gamma_n(\Sigma) \subset A^{C_n}$.
    \end{proposition}

    By \Cref{trigonalizable-subgroups-are-PT}, $\Sigma$ is $p$-by-abelian and we have a decomposition $\Sigma = UT$, where $T$ is an abelian $p'$-group and $U = O_p(\Sigma)$, and moreover $\gamma_n(U) = 1$.
    Let $\pi : \Sigma \to T$ be the natural projection.
    Let $H = [U, T]$.
    By \Cref{lem:U-decomp,lem:gamma_omega}, $U = H C_U(T)$, $H = [H,T]$, and $H = \gamma_\omega(\Sigma) = \gamma_n(\Sigma)$.
    We will prove that $H \subset A^{O_n(1)}$.

    Several times in the proof we will replace $A$ by a small power. This operation is justified by \Cref{lem:tripling}.

    \subsubsection{Central case}

    First we will prove that if $Z$ is a normal subgroup of $\Sigma$
    contained in $H$ such that $[Z, U] = C_Z(T) = 1$ and $A$ covers $H/Z$ then $A^{O_n(1)}$ covers $Z$.

    We use the idea of ``descent'' from \cite{gill-helfgott}.
    Let $\Sigma_1 = Z C_U(T) T$.
    Let $A_1 = A^3 \cap \Sigma_1$.
    Then we claim
    \begin{enumerate}[(1)]
        \item $A_1$ generates $\Sigma_1$,
        \item $A_1$ has no $K^{C_n}$-small roots.
    \end{enumerate}
    Since $A$ covers $H/Z$ we have $\Sigma = A \Sigma_1$,
    so \Cref{lem:schreier} implies (1).
    Let $A_H = A \cap H$. Then $A_H$ also covers $H/Z$, so $\Sigma = A_H \Sigma_1$ for the same reason.
    In particular $A \subset A_H \Sigma_1$,
    so $A \subset A_H(A_H A \cap \Sigma_1)$.
    Applying $\pi$, $A^\pi \subset (A_H A \cap \Sigma_1)^\pi \subset A_1^\pi$.
    Hence (2) holds.

    By \Cref{lem:U-decomp,lem:gamma_omega}, $Z = [Z, T] \times C_Z(T) = [Z, T]$ and $Z = \gamma_n(\Sigma_1)$.
    The $n$-fold commutators $[a_1, \dots, a_n]$ with $a_1, \dots, a_n \in A_1$ all lie in $C = A_1^{2^{n+1}} \cap Z$, and they normally generate $Z = \gamma_n(\Sigma_1)$ since $A_1$ generates $\Sigma_1$.
    Since $Z$ is central in $U$, $C^{\Sigma_1} = C^T$, so $\langle C^T\rangle = Z$.
    Now we can apply \Cref{section-pivoting} and we find that $Z \subset A_1^{O_n(1)} \subset A^{O_n(1)}$, as claimed.

    \subsubsection{Abelian case}

    Now suppose $V$ is an abelian normal subgroup of $\Sigma$ contained in $H$ such that $C_V(T) = 1$,
    and assume that $A$ covers $H/V$.
    We claim that $A^{O_n(1)}$ covers $H$.
    Suppose $Z \leq V$ and $Z \nsgp \Sigma$ and $A^m$ covers $H / Z$.
    Then $C_Z(T) \leq C_V(T) = 1$, so by the central case $A^{O_n(m)}$ covers $Z / [Z, U]$,
    so $A^{O_n(m)}$ covers $H / [Z, U]$.
    Iterating this $n-1$ times starting with $Z=V$ and using
    \[
        [V, \underbrace{U, \dots, U}_{n-1}] \leq \gamma_n(U) = 1,
    \]
    it follows that $A^{O_n(1)}$ covers $H$.

    \subsubsection{General case}

    Finally consider $H$ itself.
    Let $k \geq 1$ be minimal such that $\gamma_{k+1}(H) = 1$.
    Then $V = \gamma_k(H)$ is a characteristic central subgroup of $H$.
    By induction on $k$ we may assume $A^m$ covers $H / V$ for some $m \leq O_n(1)$.
    Replacing $A$ with $A^m$, we may assume $A$ covers $H / V$.
    We claim that $V \subset A^{O_n(1)}$.

    Let $W = [V, T]$.
    Since $W$ is centralized by $H$ and normalized by $C_U(T)$ and $T$,
    we have $W \nsgp \Sigma$.
    By \Cref{lem:U-decomp}, $[W, T] = W$ and $C_W(T) = 1$.
    Hence $A^{O_n(1)}$ covers $W$ by the abelian case.

    Thus we may assume $[V, T] = 1$.
    Since $H = [H, T]$, either $H$ is trivial or $k \geq 2$.
    If $k \geq 2$ the commutator induces a well-defined map
    \[
        \beta : H / \gamma_2(H) \times \gamma_{k-1}(H) / \gamma_k(H) \to V
    \]
    which is bilinear (see \cite{aschbacher}*{(8.5.4)}) and whose image generates $V$.
    Moreover, $\beta$ is $T$-invariant in the sense of \Cref{prop:bilinear},
    because $[x^t, y^t] = [x, y]^t = [x, y]$ for $x \in H$ and $y \in \gamma_{k-1}(H)$.
    Since $H \subset A\gamma_k(H)$,
    the image of $\beta$ is contained in the set $B = A^4 \cap V$,
    so $B$ generates $V$.
    On the other hand, $B$ has tripling at most $K^{11}$ by \Cref{lem:growth-in-sections}(1).
    By \Cref{lem:CHT}, $C_{H / \gamma_2(H)}(T) = 1$.
    By \Cref{lem:n^2-roots}, any $T$-simple composition factor of $H / \gamma_2(H)$ is isomorphic to a $T$-module of the form $\F_p(\chi(T))$ for some root $\chi : T \to \F$.
    Since $A$ has no $K^{C_n}$-small-roots, $H / \gamma_2(H)$ has no $T$-simple composition factor of size less than $K^{C_n}$.
    Hence, by \Cref{prop:bilinear}, $B^4 = V$, so $V \subset A^{16}$, as claimed.

    This completes the induction on $k$, and we have proved that $H \subset A^{O_n(1)}$. The proof of \Cref{prop:nsr-case} is then complete provided that $C_n$ is at least as large as this implicit constant.

    \subsection{Groups with a good section}
    \label{sec:good-section}

    The hypotheses \emph{(i)--(iv)} of the following proposition are guaranteed by \Cref{thm:main-soluble} and \Cref{prop:trig+}.
    The hypothesis \emph{(v)} on the other hand is mildly restrictive, and will be removed in the next (and final) section.

    \begin{proposition}
        \label{prop:good-section-case}
        Let $\F$ be a finite field.
        Let $A \subset \GL_n(\F)$ be a finite, nonempty, symmetric, $K$-tripling set, and let $G = \langle A\rangle$.
        Assume there is a section $\Sigma = \Gamma / P$ where $P \csgp \Gamma \nsgp G$ such that
        \begin{enumerate}[(i)]
            \item $|A \Gamma / \Gamma| \leq K^{O_n(1)}$,
            \item $\Gamma / P$ is soluble,
            \item $P \subset A^6$,
            \item $\Sigma = \tr(B, \Sigma)$ for some trigonalizable subgroup $B \leq \Gamma$,
            \item $G$ acts trivially on $\Sigma / O_p(\Sigma)$.
        \end{enumerate}
        Then there is a normal subgroup $\Delta \nsgp G$ such that $P \leq \Delta \leq \Gamma$ and
        \begin{enumerate}[(1)]
            \item $|A\Delta / \Delta| \le K^{O_n(1)}$, and
            \item $\gamma_n(\Delta) \subset A^{O_n(1)}$.
        \end{enumerate}
    \end{proposition}

    \begin{proof}
    First we shrink $\Sigma$ (and $\Gamma$) until there are no small roots, as follows.
    Let $\chi_1, \dots, \chi_m \in \roots^*(\Sigma)$ be the $R$-small roots for $\tr (A^2, \Sigma)$, for some value of $R$.
    Note that $m \leq n^2$ by \Cref{lem:n^2-roots}.
    Let $\Sigma_m = \bigcap_{i=1}^m \ker \chi_i \nsgp \Sigma$.
    Then
    \[
        |\tr(A^2, \Sigma / \Sigma_m)|
        \leq \prod_{i=1}^m |\chi_i(\tr(A^2, \Sigma))|
        \leq R^m.
    \]
    Now if $\chi \in \roots^*(\Sigma)$ is an $R$-small root for $\tr(A^2, \Sigma_m)$ then, since $\tr(A^2, \Sigma)$ is covered by $R^m$ cosets of $\Sigma_m$, and therefore by $R^m$ translates of $\tr(A^2, \Sigma)^2 \cap \Sigma_m \subset \tr(A^4, \Sigma_m)$ (see \Cref{lem:covering-observation}), it follows from \Cref{lem:growth-in-sections}(3) that
    \[
        |\chi(\tr (A^2, \Sigma))|
        \leq R^m |\chi(\tr(A^4, \Sigma_m))|
        \leq K^7 R^m |\chi(\tr(A^2, \Sigma_m))|
        \leq K^7 R^{m+1},
    \]
    so $\chi$ is a $K^7 R^{m+1}$-small root for $\tr(A^2, \Sigma)$.
    By the pigeonhole principle we can choose $R$ so that $K^{C'_n} \leq R \leq K^{O_n(1)}$, where $C'_n = 15C_n + 3$ and $C_n$ is the constant in \Cref{prop:nsr-case}, and such that there is no $\chi \in \roots^*(\Sigma)$
    such that
    \[
        R < |\chi(\tr(A^2, \Sigma))| \leq K^7 R^{m+1}.
    \]
    For this value of $R$, it follows that $\tr(A^2, \Sigma_m)$ has no $R$-small roots $\chi \in \roots^*(\Sigma)$.
    Since $O_p(\Sigma_m) = O_p(\Sigma)$, we have $\roots^*(\Sigma_m) \subset \roots^*(\Sigma)$, so a fortiori $\tr(A^2, \Sigma_m)$ has no $R$-small roots in $\roots^*(\Sigma_m)$.
    Write $\Sigma_m = \Gamma_m / P$.
    Then $\Gamma_m \nsgp \Gamma$, and in fact $\Gamma_m \nsgp G$ since $G$ acts trivially on $\Sigma / O_p(\Sigma)$.
    From \Cref{lem:orbit-stab}(2),
    \[
        |A\Gamma_m / \Gamma_m| \leq |A\Gamma/\Gamma| |\tr(A^2, \Sigma/\Sigma_m)|
        \leq K^{O_n(1)}.
    \]
    Hence we may replace $\Gamma$ with $\Gamma_m$,
    and thus we may assume $\tr(A^2, \Sigma)$ has no $K^{C'_n}$-small roots.

    Next we replace $\Gamma$ with the preimage of $\gen{\tr(A^2, \Sigma / O_p(\Sigma))}$, ensuring that $\tr(A^2, \Sigma / O_p(\Sigma))$ generates $\Sigma / O_p(\Sigma)$.
    This does not change the value of $|\chi(A)|$ for any root $\chi$,
    nor does it change the value of $|A\Gamma/\Gamma|$, by \Cref{lem:covering-observation}.
    Again this does not compromise normality of $\Gamma$ in $G$ because $G$ acts trivially on $\Sigma / O_p(\Sigma)$.

    Let $U = O_p(\Sigma)$.
    For $i = 1, 2$ define
    \begin{align*}
        \Sigma_i &= \langle \tr(A^{2i}, \Sigma) \rangle,\\
        U_i &= \Sigma_i \cap U,\\
        H_i &= \gamma_n(\Sigma_i).
    \end{align*}
    By \Cref{trigonalizable-subgroups-are-PT}, $\Sigma$ is $p$-by-abelian and $\gamma_n(U) = 1$.
    By Schur--Zassenhaus, $\Sigma_1 = U_1 T$ for an abelian $p'$-group $T$.
    Since $\tr(A^2, \Sigma / U)$ generates $\Sigma / U$ we also have $\Sigma = UT$
    as well as $\Sigma_2 = \Sigma_2 \cap UT = U_2 T$ by the modular law.
    By \Cref{lem:U-decomp,lem:gamma_omega}, $H_i = [U_i, T] = [H_i, T]$.
    Also, note that
    \[
        \Sigma_1^a \leq \Sigma_2 \qquad (a \in A).
    \]

    Let $V = H_2 / H_2'$.
    Since $[H_2, T] = H_2$ it follows that $[V, T] = V$ and, by \Cref{lem:U-decomp}, $C_V(T) = 1$.
    Similarly for any subgroup $L \le V$ we have $C_L(T) \le C_V(T) = 1$ and, by \Cref{lem:U-decomp}, $L = [L, T]$.
    Applying this to $L = \langle \tr (A^2, V) \rangle$, we have
    \[
        \tr(A^2, V) \subset L = [L, T] \leq \tr([U_1, T], V) = \tr(H_1, V) = H_1 H_2' / H_2'.
    \]
    By \Cref{prop:nsr-case} applied to $\tr(A^4, \Sigma)$,
    which has tripling at most $K^{15}$ by \Cref{lem:growth-in-sections}
    and no $K^{15C_n}$-small roots,
    we have $H_2 \subset \tr(A^{15C_n}, H_2)$.
    Hence, by \Cref{lem:growth-in-sections},
    \[
        [H_2 : H_1 H_2'] \leq \frac{|\tr(A^{15C_n}, H_2)|}{|\tr(A^2, H_2)|} \leq K^{15C_n+3}.
    \]
    By \Cref{lem:n^2-roots}, any $T$-simple quotient of $H_2 / (H_1 H_2')$ is isomorphic to a $T$-module of the form $\F_p(\chi(T))$ for some root $\chi : T \to \F$. Since $H_2 = [H_2, T]$, we have $W = [W, T]$, so $\chi$ must be nontrivial.
    On the other hand the above bound shows that $|W| \le K^{15 C_n+3}$.
    Since $T$ has no $K^{15 C_n+3}$-small roots, there is no such $W$.
    Thus $H_2 = H_1 H_2'$.
    By the Burnside basis theorem it follows that $H_2 = H_1$.

    Let $H = H_1 = H_2$.
    Hence $\Sigma_i = H C_{U_i}(T) T$ by \Cref{lem:U-decomp}.
    Since $H^a = \gamma_n(\Sigma_1^a) \leq \gamma_n(\Sigma_2) = H$ for $a \in A$, it follows that $H$ is normalized by $G$.
    At last define
    \[
        \Omega = \Sigma_1 C_U(T) = \Sigma_2 C_U(T) = H C_U(T) T.
    \]
    For $a \in A$ we have $T^a \leq \Sigma_1^a \leq \Sigma_2 \leq \Omega$.
    Since $\Omega = TC_U(T)H$, it follows that $T^a = T^h$ for some $h \in H$ by Schur--Zassenhaus.
    Hence also
    \[
        C_U(T)^a = C_U(T)^h \leq C_U(T) H \leq \Omega.
    \]
    Thus $\Omega$ is normalized by $G$.
    By \Cref{lem:gamma_omega} (and recalling $\gamma_n(U) = 1$), we have $\gamma_n(\Omega) = [H C_U(T), T] = [H, T] = H$,
    and we saw earlier that $H = H_2$ is covered by $A^{15 C_n}$.

    Let $\Delta$ be the preimage of $\Omega$ in $\Gamma$.
    Recall that $A$ is covered by $|A\Gamma/\Gamma|$ translates of $A^2 \cap \Gamma$ by \Cref{lem:covering-observation}.
    Therefore since $\tr(A^2, \Sigma) \subset \Omega$, we have $|A\Delta/\Delta| = |A\Gamma/\Gamma|$.
    Now since $\gamma_n(\Omega)$ is covered by $A^{15 C_n}$
    and $P \subset A^6$
    it follows that $\gamma_n(\Delta) \subset A^{15 C_n+6}$.
    This completes the proof.
    \end{proof}

    \subsection{Creating a good section}

    Finally, let $A$ be as in \Cref{thm:main}: $A$ is a nonempty, symmetric, $K$-tripling subset of $\GL_n(\F)$.
    By \Cref{prop:main-finitization} we may assume $\F$ is finite.
    Let $G = \langle A \rangle$.
    By \Cref{thm:main-soluble}, there is a soluble section $\Sigma = \Gamma / P$ where $P \csgp \Gamma \nsgp G$ such that \emph{(i)--(iii)} of \Cref{prop:good-section-case} are satisfied.
    Moreover $P$ is perfect, soluble-by-$\Lie^*(p)$, and contained in a translate of $A^3$.
    By \Cref{prop:trig+} (and replacing $\Sigma$ with $\Sigma_0$), we can assume \emph{(iv)} holds too,
    and moreover we can assume that $G_0 = C_G(\Sigma / O_p(\Sigma))$ has index at most $n!$ in $G$.

    Let $m = [G : G_0]$, so $m \leq n!$.
    Since $A$ generates $G$, we have $G = A^m G_0$, so $G_0$ is generated by $A_0 = A^{3m} \cap G_0$ by \Cref{lem:schreier}.
    By \Cref{lem:growth-in-sections}(1), $A_0$ has tripling $K^{O(m)}$.
    By \Cref{lem:growth-in-sections}(2), $|A_0\Gamma/\Gamma| \leq K^{O(m)} |A\Gamma/\Gamma| \le K^{O_n(1)}$.
    Since $P$ acts trivially on $\Sigma$ we have $P \subset A^6 \cap G_0 \subset A_0^6$.
    By \Cref{trigonalizable-subgroups-are-PT}, $\Sigma / O_p(\Sigma)$ is abelian, so $\Gamma \le G_0$.
    Hence the hypotheses of \Cref{prop:good-section-case} hold with $(A_0, G_0, K^{O(m)})$ in the role of $(A, G, K)$.
    Thus there is $\Delta \nsgp G_0$ such that $P \leq \Delta \leq \Gamma$
    and $|A_0\Delta/\Delta| \leq K^{O_n(1)}$ and $\gamma_n(\Delta) \subset A_0^{O_n(1)}$.
    Note also that $\gamma_n(\Delta) / P$ is a $p$-group since $\gamma_n(\Gamma) / P \cong \gamma_n(\Sigma)$ and $\Sigma / O_p(\Sigma)$ is abelian, as noted.

    Since $A$ is covered by $m$ cosets of $G_0$, it is covered by $m$ translates of $A^2 \cap G_0 \subset A_0$ (\Cref{lem:covering-observation}), so $|A\Delta/\Delta| \leq m |A_0\Delta/\Delta| \leq K^{O_n(1)}$.

    The only remaining issue is that while $\Delta$ is normal in $G_0$ it may not be normal in $G$.
    But since $[G : G_0] = m$, there are at most $m$ conjugates of $\Delta$ in $G$, and since $G = A^m G_0$ each of them has the form $\Delta^a$ for some $a \in A^m$.
    Let $\Delta_0 \nsgp G$ be their intersection.
    If $a \in A^m$ then
    \[
        |A \Delta^a / \Delta^a|
        = |A^{a^{-1}} \Delta / \Delta|
        \leq |A^{2m+1} \Delta / \Delta|
        \leq K^{2m+1} |A\Delta / \Delta|
    \]
    by \Cref{lem:growth-in-sections}(2).
    It follows that $|A \Delta_0 / \Delta_0| \leq (K^{2m+1} |A\Delta / \Delta|)^m \leq K^{O_n(1)}$,
    and obviously $\gamma_n(\Delta_0) \subset \gamma_n(\Delta) \subset A^{O_n(1)}$.
    Moreover $P \leq \Delta_0$ and $\gamma_n(\Delta_0) / P$ is a $p$-group.
    This finally completes the proof of \Cref{thm:main}.

\section{Diameter of quasirandom groups}
\label{sec:application}

As promised in the introduction, here we show that
one of our key new ideas in the proof of \Cref{thm:main},
namely the ``affine conjugating trick''
(\Cref{lem:affine-conjugating-trick-intro})
has another application to growth-type questions.

\begin{theorem} [\Cref{thm:polylog-dimeter-intro}~restated]
  \label{thm:polylog-dimeter}
  For each positive integer $n$ there are positive numbers
  $K=K(n)$ and $c=c(n)$
  with the following property.
  Let $\F$ be a field and $G\le\GL_n(\F)$ be a $K$-quasirandom finite
  subgroup. Then the Cayley graph of $G$ with respect to any
  generating set has diameter at most $\br{\log|G|}^c$.
\end{theorem}

This theorem is sharp in the following sense.
In \cite{PS-JAMS}*{Example~75} a perfect subgroup $G$ of $\SL_5(\F_q)$
is constructed which has diameter $|G|^c$ for a constant $c > 0$.
Slightly modifying this example, one can obtain subgroups
$G \le \SL_n(\F_q)$ with $\deg_\C(G)\ge n-1$ and diameter $|G|^{c_n}$.

The proof of \Cref{thm:polylog-dimeter-intro}
is based on \Cref{lem:affine-conjugating-trick-intro}
and on a boundedness property of finite linear groups
(\Cref{thm:linear-actions-are-bounded})
which is  of independent interest.
The proof of \Cref{thm:linear-actions-are-bounded}
in turn uses a result of McNinch concerning
connected algebraic groups acting on connected unipotent groups,
and on an important result of Steinberg on representations of
finite simple groups of Lie type.

\begin{definition}
    Let $G$ be a group acting algebraically on a connected unipotent
    group $U$ over an algebraically closed field.
    The action is \emph{linearizable}
    if there is a $G$-invariant normal chain of connected
    closed subgroups ${1=U_0\nsgp U_1\nsgp\dots\nsgp U_k=U}$
    such that each quotient $U_j/U_{j-1}$ is
    $G$-equivariantly isomorphic to a vector space with a linear $G$-action.
\end{definition}

\begin{lemma} \label{lem:linearizable-action}
    Let $\F$ be an algebraically closed field of characteristic $p>0$,
    let $U_0\le\GL_n(\F)$ be a unipotent subgroup,
    and let $G\le\GL_n(\F)$ be a subgroup normalizing $U_0$.
    Then there is a $G$-invariant connected closed unipotent subgroup
    $U \le \GL_n(\F)$ containing $U_0$
    and
    $G$ has a subgroup $H$ of index $O_n(1)$
    whose action on $U$ is linearizable.
\end{lemma}

\begin{proof}
    Let $M_n(\F)$ denote the linear space of $n$-by-$n$ matrices over
    $\F$.
    Then $\GL_n(\F)$ acts linearly on $M_n(\F)$ via conjugation.
    Let $N \le M_n(\F)$ be the linear span of $\{u-1: u \in U_0\}$.
    Then $N$ is a $G$-invariant nilpotent subalgebra of $M_n(\F)$
    and $U = 1 + N$ is a $G$-invariant connected closed unipotent subgroup.

    Let $S$ be the normalizer of $U$ in $\GL_n(\F)$. We claim that $[S:S^\circ] = O_n(1)$.
    Observe that $\opr{Ad}(S)$ is just the intersection of $\opr{Ad}(\GL_n(\F))$ with the stabilizer of $N$ in $\GL(M_n(\F))$,
    which is itself the intersection of $\GL(M_n(\F))$ with a linear subspace (of codimension $d (n^2 - d)$, where $d = \dim N$).
    Therefore by Bezout's theorem the number of components of $S$ is
    bounded by the degree of
    $\opr{Ad}(\GL_n(\F))=\opr{Ad}(\SL_n(\F))$,
    which is at most $n^{n^2}$ (see \cite{angelo-MO}).

    Let $H = G \cap S^\circ$.
    Since $G \le S$ we have $[G:H] \le [S:S^\circ] = O_n(1)$.

    The result of McNinch~\cite{mcninch2014linearity}
    implies that the $S^\circ$-action on $U$ is linearizable.
    Hence the $H$-action on $U$ is linearizable.
\end{proof}

\begin{definition} \label{def:B-bounded-action-onp-group}
    Let $G$ be a finite group acting on a $p$-group $P$.
    Assume that $G/O_p(G)$ is a central product of quasisimple groups
    $S_1,\dots,S_l$.
    We say that a $G$-invariant normal chain
    $1=P_0\nsgp P_1\nsgp\dots\nsgp P_k=P$
    of length $k$
    is \emph{$B$-bounded} for some $B>0$
    if
    \begin{enumerate} [(a)]
        \item
        each quotient $P_j/P_{j-1}$ is elementary abelian,
        \item
        as an $\F_pG$-module $P_j/P_{j-1}$ is isomorphic
        to the direct sum of isomorphic copies of some irreducible
        $\F_pG$-module $W_j$,
        \item
        $W_j$ has a tensor product decomposition
        $W_j\cong X_{j,1}\otimes_{\F_p}\cdots\otimes_{\F_p}X_{j,l}$
        where each $X_{j,i}$ is an irreducible $\F_pS_i$-module,
        \item \label{item:trivial-or-narrow}
        either $X_{j,i}\cong\F_p$ with trivial $S_i$-action, or
        \[
            X_{j,i} = (S_i x)^{\pm B}
            \qquad
            \text{for all nonzero}~x\in X_{j,i}.
        \]
    \end{enumerate}
\end{definition}

\begin{lemma} \label{lem:sections-of-B-bounded-action}
    Let $G$ be a finite group acting on a $p$-group $P$.
    If $P$ has a $B$-bounded $G$-invariant normal chain of length $k$
    then each $G$-invariant section of $P$
    has a $B$-bounded $G$-invariant normal chain of length $k$.
\end{lemma}
\begin{proof}
Clear.
\end{proof}

\begin{theorem} \label{thm:linear-actions-are-bounded}
    For all $n>0$ there is an integer $K=K(n)>0$ with the following property.
    Let $\F$ be a field of characteristic $p>0$,
    let $G\le\GL_n(\F)$ be a finite subgroup
    and let $P<\GL_n(\F)$ be a $p$-group normalized by $G$.
    If $\deg_\C(G)\ge K$
    then any $G$-invariant section of $P$
    has an $O_n(1)$-bounded $G$-invariant normal chain of length less
    than $n^2$.
\end{theorem}
\begin{proof}
    If we choose $K$ large enough then by Weisfeiler's theorem
    $G/O_p(G)$  is isomorphic to a central product of at most
    $n$ quasisimple groups of Lie type of characteristic $p$
    with Lie rank at most $n$
    (see \Cref{thm:weisfeiler,lem:lies}).
    Let $S_1,\dots,S_l$ denote the factors,
    and let $S$ be their direct product.

    Let $\overline\F$ be the  algebraic closure of $\F$.
    \Cref{lem:linearizable-action}
    gives us a subgroup $H\le G$ of index $O_n(1)$
    and a $G$-invariant connected unipotent subgroup
    $U<\GL_n(\overline\F)$ containing $P$ with linearizable $H$-action
    (it is well known that $p$-subgroups of $\GL_n(\F)$ are
    unipotent).
    By choosing $K$ large enough we make sure that $G=H$.
    We shall prove that $U$ has an $O_n(1)$-bounded $G$-invariant normal
    chain of length less than $n^2$.

    Since the action of $G$ on $U$ is linearizable, there is a $G$-invariant normal chain
    $1=U_0\nsgp U_1\nsgp\dots\nsgp U_k=U$
    of connected closed subgroups
    such that
    each quotient $V_j=U_j/U_{j-1}$ a isomorphic to a vector space over $\bar\F$ with a linear $G$-action.
    We can refine this chain so that each $V_j$ is
    an irreducible $\overline\F G$-module.
    The length of this chain is $k\le\dim(U)<n^2$.
    We shall prove that the normal chain $U_0\nsgp\dots\nsgp U_k$
    in $O_n(1)$-bounded.

    Since $V_j$ is irreducible, $O_p(G)$ acts on it trivially,
    so it is an irreducible $\overline\F S$-module.
    Therefore it can be written as a tensor product
    \[
    V_j\cong
    Y_{j,1}\otimes_{\overline\F}\dots\otimes_{\overline\F}Y_{j,l}
    \]
    where each $Y_{j,i}$ is an irreducible $\overline\F S_i$-module.

    Let $\F_{q_i}$ be the defining field of $S_i$.
    By a result of Steinberg \cite{steinberg1963representations}, there is an $\F_{q_i}S_i$ module $M_{j,i}$ such that
    \[
    Y_{j,i}\cong M_{j,i}\otimes_{\F_{q_i}}\overline\F,
    \]
    and $\dim_{\F_{q_i}}(M_{j,i})<n^2$.
    Note that in Steinberg's result not all simple groups are
    covered:  odd-dimensional unitary groups and Ree groups of type
    $^2G_2$  are excluded.
    The reason for this exclusion is that one needs to lift projective
    representations to ordinary representations
    of  the group he called  $\Gamma^1_q$
    constructed from the simply connected algebraic group.
    Since we now know the Schur multipliers of these groups, we can
    include them as well, with a small number of exceptions with bounded
    size. We exclude those exceptions from the $S_i$ by making $K$ large enough.

    Each $M_{j,i}$ as an $\F_pS_i$-module must be the direct sum of
    isomorphic copies of some irreducible $\F_pS_i$-module $X_{j,i}$.
    Therefore $V_j$ as an $\F_p S$-module is the direct sum of isomorphic copies of $X_{j,1} \otimes_{\F_p} \cdots \otimes_{\F_p} X_{j,l}$.

    Finally suppose that $S_i$ acts nontrivially on $X_{j,i}$ and $x \in X_{j,i}$.
    The bound $\dim_{\F_{q_i}}(M_{j,i}) < n^2$ implies that $|M_{j,i}| = |S_i|^{O_n(1)}$ and therefore also $|X_{j,i}| = |S_i|^{O_n(1)}$.
    Let $A = (S_i x)^{\pm1} \subset X_{j,i}$.
    Then $A$ is a symmetric $S_i$-invariant generating set for $X_{j,i}$.
    Now we apply the affine conjugating trick to $A \subset X_{j,i} \rtimes S_i$ with $K = |S_i|^{c/n}$ (see \Cref{landazuri--seitz}).
    Let $k \ge 0$ be minimal such that $|A^{3^{k+1}}| < K |A^{3^k}|$.
    Then
    \[
        |S_i|^{ck/n} \le K^k |A| \le |A^{3^k}| \le |X_{j,i}| = |S_i|^{O_n(1)},
    \]
    which implies that $k = O_n(1)$.
    Since $|A^{3^k}|$ has tripling less than $K$, \Cref{lem:affine-conjugating-trick-intro} implies that $A^{3^k \cdot 14} = X_{j,i}$.
    Thus $X_{j,i} = A^{O_n(1)}$.

    This proves that the normal chain $U_0\nsgp\cdots\nsgp U_k$
    is $O_n(1)$-bounded.
    By \Cref{lem:sections-of-B-bounded-action}
    the $G$-action on
    each $G$-invariant section of $P$ has an $O_n(1)$-bounded
    $G$-invariant normal chain of length less than $n^2$.
\end{proof}

\begin{lemma} \label{lem:application-induction-step}
    Let $G$ be a finite group acting on an abelian $p$-group $V$ such that $V = [V, G]$.
    Assume that $G / O_p(G)$ is $d$-generated and a central product of quasisimple groups.
    Suppose that $V$ has a $B$-bounded $G$-invariant normal chain of length $k$.
    If $A\subseteq V$ is a $G$-invariant symmetric generating set containing $1$
    then $A^{O(Bd\log|V|)^k}=V$.
\end{lemma}

\begin{proof}
    Let $1=V_0\nsgp V_1\nsgp\cdots\nsgp V_k=V$ be
    a $B$-bounded normal chain.
    We prove the statement by induction on $k$.

    If $k=0$ then $V$ is trivial and the statement clearly holds.

    For the induction step we assume that the statement holds for the
    group $V/V_1$. This gives us an exponent
    $c =  O(Bd \log |V|)^{k-1}$ such that $A^c$ covers $V/V_1$.
    According to
    \Cref{def:B-bounded-action-onp-group}\ref{item:trivial-or-narrow}
    we have two possibilities.

    Suppose first that $V_1$ is a trivial $\F_pG$-module.
    Let $g_1, \dots, g_d$ be elements of $G$ whose images generate $G/O_p(G)$.
    Then
    \[
        \{[v_1, g_1] \cdots [v_d, g_d] : v_1, \dots, v_d \in V\} =
        [V, G] = V
    \]
    by \Cref{lem:rhemtulla}.
    Since $[V_1,G]=1$, $[v_i,g_i]$ depends only on the
    image of $v_i$ in $V/V_1$. Therefore
    \[
      \{[v_1, g_1] \cdots [v_d, g_d] : v_1, \dots, v_d \in A^c\} = V.
    \]
    Hence $A^{2dc}=V$ and  the induction step is complete in this case.

    Consider now the case when $V_1$ is a nontrivial $\F_pG$-module.
    Let $S$ be one of the quasisimple factors of $G/V$ which acts on
    $V_1$ nontrivially. By
    \Cref{def:B-bounded-action-onp-group}\ref{item:trivial-or-narrow},
    as an $\F_pS$-module $V_1$ is the direct sum of isomorphic copies
    of an irreducible $\F_pS$-module $X$ such that
    \[
      X = (Sx)^{\pm B}
      \quad\quad
      \text{for all nonzero }x\in X.
    \]
    By \Cref{lem:schreier} (Schreier's lemma),
    $A_1=A^{3c}\cap V_1$ is a symmetric generating set of $V_1$.
    Each element $a\in A_1$ generates an $\F_pS$-submodule $X_a\le V_1$
    isomorphic to $X$, so $X_a = (Sa)^{\pm B}$.
    Hence $A_1^B$ contains $X_a$ for every $a \in A_1$.
    Since $V_1$ is the sum of the submodules $X_a$, it is the sum of some $\log_p |V_1|$ of them.
    Therefore $V_1=A_1^{B \log_p |V_1|}$,
    and so $V \subset A^{c + c B \log_p |V_1|} \subset A^{O(Bd \log|V|)^k}$.
    The induction step is complete in this case too.
\end{proof}

\begin{proof}[Proof of \Cref{thm:polylog-dimeter-intro}
    ($=$\Cref{thm:polylog-dimeter})]
    If $\F$ has characteristic $0$ then by Jordan's theorem $G$ has an
    abelian subgroup of index $O_n(1)$. By choosing $K(n)$ larger
    than this bound we can ensure that $G$ is abelian and hence $G$ is trivial.
    The statement holds in this case.

    Hence assume that $\F$ has characteristic $p>0$.
    By \Cref{thm:weisfeiler,cor:weisfeiler}, $G/O_p(G)$ is a central product of quasisimple groups of Lie type of characteristic $p$,
    $P = [G, \Sol(G)]$ is a $p$-subgroup such that $P = [P, G]$,
    and $|\Sol(G) / P| \le (2n+1)^n$.
    By \Cref{lem:lies}, $G / \Sol(G) \in \Lie^n(p)$.
    In particular, $G / P$ is $O_n(1)$-generated.

    Let $A\subseteq G$ be a symmetric generating set.
    Since $G / \Sol(G) \in \Lie^n(p)$ and $\Sol(G) / P$ has bounded order, $G / P$ has poly-logarithmic diameter. Hence $A^\ell$ covers $G/P$, where $\ell = c_1 (\log |G|)^{c_2}$ and $c_1 = c_1(n)$ and $c_2 = c_2(n)$.

    By Schreier's lemma (\Cref{lem:schreier}), $A^{3\ell}\cap P$ generates $P$.
    Therefore $A^{5\ell}$ contains a set $B \subset P$ whose image in $P / [P, P]$ is a symmetric $G$-invariant generating set containing $1$.

    If $K$ is large enough then by \Cref{thm:linear-actions-are-bounded}
    the section $P/[P,P]$ has an $O_n(1)$-bounded $G$-invariant normal
    chain of length less than $n^2$.
    Applying \Cref{lem:application-induction-step} to $G$ and $P/[P,P]$,
    we obtain that
    $X = B^{O_n(\log|P|)^{n^2}}$ maps surjectively onto $P/[P,P]$.

    The nilpotency class of $P$ is at most $n-1$.
    Consider the set $X_i \subset X^{2^{i+1}}$ of all weight-$i$ left-normed commutators $[x_1, \dots, x_i]$, where $x_1, \dots, x_i \in X$.
    The image of $X_i$ in $\gamma_i(P) / \gamma_{i+1}(P)$ is a generating set and a union of subgroups, so $X_i^{O(\log |P|)}$ covers $\gamma_i(P) / \gamma_{i+1}(P)$.
    This shows that $X^{O(2^i \log |P|)}$ covers $\gamma_i(P) / \gamma_{i+1}(P)$ for each $i=1,\dots, n-1$, so $X^{O(2^n \log |P|)}$ covers $P$.
    Thus $A^{O_n(\log |G|)^{O_n(1)}}$ covers $G$.
\end{proof}

\bibliography{refs}
\end{document}